\documentclass{amsart}
\usepackage{amsmath,amssymb, amsbsy}
\usepackage[dvips]{graphicx}
\usepackage{color}
\usepackage[latin1]{inputenc}
\usepackage[active]{srcltx}
\usepackage{wrapfig}
\usepackage{graphicx}
\usepackage[usenames,dvipsnames]{pstricks}
\usepackage{epsfig}
 \usepackage{pst-grad} 
 \usepackage{pst-plot} 
\usepackage{epic}

\newcommand{\e }{\varepsilon }

\renewcommand{\O }{\Omega}

\renewcommand{\d }{\delta}
\renewcommand{\l }{\lambda }

\long\def\salta#1{\relax}

\newcommand{\re}{{I\!\!R}}

\newcommand{\dyle}{\displaystyle}
\newcommand{\ene}{{I\!\!N}}

\newcommand{\io}{\int\limits_\O}

\newcommand{\car}{{\raise2pt\hbox{$\chi$}}}

\newcommand{\dint}{\dyle\int}

\renewcommand{\a }{\alpha }

\renewcommand{\d }{\delta }

\renewcommand{\l }{\lambda }
\renewcommand{\L }{\Lambda }

\newcommand{\s }{\sigma }

\renewcommand{\O }{\Omega }

\newcommand{\RR}{\mathbb R}

\renewcommand{\le}{\leqslant}
\renewcommand{\leq}{\leqslant}
\renewcommand{\ge}{\geqslant}
\renewcommand{\geq}{\geqslant}

\def\Int{\displaystyle\int}

\newtheorem{Theorem}{Theorem}[section]

\newtheorem{Definition}[Theorem]{Definition}

\newtheorem{Proposition}[Theorem]{Proposition}
\newtheorem{remark}[Theorem]{Remark}

\begin{document}

\title[Optimal mixed Hardy constant]{Attainability of the fractional Hardy constant with nonlocal mixed boundary conditions. Applications}
\author[B. Abdellaoui, A. Attar, A. Dieb, I. Peral]{Boumediene Abdellaoui, Ahmed Attar, Abdelrazek Dieb,  Ireneo Peral }
\thanks{ This work is partially supported by project  MTM2013-40846-P and MTM2016-80474-P, MINECO, Spain.} 

\address{\hbox{\parbox{5.7in}{\medskip\noindent {B. Abdellaoui, A. Attar: Laboratoire d'Analyse Nonlin\'eaire et Math\'ematiques
Appliqu\'ees. \hfill \break\indent D\'epartement de Math\'ematiques, Universit\'e Abou Bakr Belka\"{\i}d, Tlemcen, \hfill\break\indent Tlemcen 13000,
Algeria.}}}}
\address{\hbox{\parbox{5.7in}{\medskip\noindent {A. Dieb: Laboratoire d'Analyse Nonlin\'{e}aire et Math\'{e}matiques Appliqu\'{e}es.
\hfill \break\indent Universit\'{e} Abou Bakr Belka\"{i}d, Tlemcen, Tlemcen 13000, Algeria, \\
and  D\'{e}partement de Math\'{e}matiques, Universit\'{e} Ibn Khaldoun, Tiaret, Tiaret 14000, Algeria.}}}}
\address{\hbox{\parbox{5.7in}{\medskip\noindent{I. Peral: Departamento de Matem\'aticas,\\ Universidad Aut\'onoma de Madrid,\\
        28049, Madrid, Spain. \\[3pt]
        \em{E-mail addresses: }\\{\tt boumediene.abdellaoui@inv.uam.es, \tt ahm.attar@yahoo.fr, \tt dieb\_d@yahoo.fr, \tt ireneo.peral@uam.es         }.}}}}

\date{}



\begin{abstract}
The first goal of this paper is to study necessary and sufficient conditions to obtain the attainability of the \textit{fractional Hardy inequality }
$$
\Lambda_{N}\equiv\Lambda_{N}(\O):=\inf_{\{\phi\in \mathbb{E}^{s}(\Omega, D), \phi\neq
0\}}\dfrac{\frac{a_{d,s}}{2}\displaystyle\int_{\mathbb{R}^d}\int_{\mathbb{R}^d}\dfrac{|\phi(x)-\phi(y)|^{2}}{|x-y|^{d+2s}}dx dy}{\displaystyle\io
\frac{\phi^2}{|x|^{2s}}\,dx},
$$
where  $\O$ is a bounded domain of $\RR^d$,  $0<s<1$, $D\subset \mathbb{R}^d\setminus \Omega$ a nonempty open set and
$$\mathbb{E}^{s}(\Omega,D)=\left\{ u \in H^s(\RR^d):\, u=0 \text{ in } D\right\}.$$
The second aim of the paper is to study the  \textit{mixed Dirichlet-Neumann boundary problem} associated to the minimization problem and  related properties;
precisely,  to study semilinear elliptic problem  for the \textit{fractional laplacian},  that is,
$$P_{\lambda} \, \equiv
\left\{
\begin{array}{rcll}
(-\Delta)^s u &= & \lambda \dfrac{u}{|x|^{2s}} +u^p & {\text{ in }}\O,\\
   u & > & 0 &{\text{ in }} \O, \\
  \mathcal{B}_{s}u&:=&u\chi_{D}+\mathcal{N}_{s}u\chi_{N}=0  &{\text{ in }}\RR^{d}\backslash \O, \\
\end{array}\right.
$$
with  $N$ and $D$ open sets in $\RR^{d}\backslash\Omega$ such that $N \cap D=\emptyset$ and $\overline{N}\cup \overline{D}= \re^{d}\backslash\Omega$, $d>2s$, $\l>
0$ and $0<p\le 2_s^*-1$, $2_s^*=\frac{2d}{d-2s}$.  We emphasize that the nonlinear term can be critical.

The operators $(-\Delta)^s $, fractional laplacian,  and $\mathcal{N}_{s}$,  nonlocal Neumann condition, are defined below in \eqref{frac lap} and \eqref{Nmn def}
respectively.
\end{abstract}

\maketitle

\rightline{\textit{ To Rafa de la Llave in his 60th birthday, with our best wishes}.}


\section{Introduction}\label{sec:intro}
The problems studied in this paper are motivated by some recent results that we summarized below.

In first place we consider the classical Hardy inequality  proved in \cite{He} (see also \cite{B, FLS, SW, Y}).

\textbf{{Theorem.}}\,{\it (Fractional Hardy inequality).} \textit{Assume that $s\in (0,1)$ is such that $2s<d$, then for all $u\in
\mathcal{C}^{\infty}_{0}(\re^d)$, the following inequality holds,
\begin{equation}\label{Hardyint}
\frac{a_{d,s}}{2}\displaystyle\int_{\mathbb{R}^d}\int_{\mathbb{R}^d}{\frac{|u(x)-u(y)|^2}{|x-y|^{d+2s}}}\, dx\, dy\equiv \dint_{\re^d} \,|\xi|^{2s}
|\hat{u}|^2\,d\xi\geq \Lambda\,\dint_{\re^d} |x|^{-2s} u^2\,dx,
\end{equation}
where $\hat{u}$ is the Fourier transform of $u$ and
\begin{equation}\label{bestC}
\Lambda= 2^{2s}\dfrac{\Gamma^2(\frac{d+2s}{4})}{\Gamma^2(\frac{d-2s}{4})} .
\end{equation}
Moreover the constant $\Lambda$ is optimal and is not attained.}

The optimal constant defined in \eqref{bestC} coincides for every  bounded domain $\Omega$ containing the pole of the Hardy potential. More precisely, if $0\in
\Omega$, then for all $u\in \mathcal{C}_0^\infty(\Omega)$ we have
\begin{equation}\label{hardy}
\frac{a_{N,s}}{2}\displaystyle\int\int_{\mathcal{D}_{\Omega}}{\frac{|u(x)-u(y)|^2}{|x-y|^{d+2s}}}\, dx\, dy\geq \Lambda\int_{\Omega}{\frac{u^2}{|x|^{2s}}\,dx},
\end{equation}
where $$\mathcal{D}_{\Omega}= (\RR^d \times \RR^d ) \setminus( \O^c\times \O^c).$$ The optimality of $\Lambda$ here follows by a scaling argument.

The others starting points for the problems considered in this work,  are some results obtained in  the articles \cite{drv}, \cite{LMPAS}  and \cite{DMPS}.

In  \cite{drv}  the authors consider a natural Neumann condition in the sense that Gauss and Green integration by parts formulas hold for such condition.    More
precisely, if $\Omega$ is a bounded open set in $ \mathbb{R}^d$ with suitable  regularity, then the Neumann problem for the fractional laplacian take the form,
\begin{equation}\label{N}
\left\{
\begin{array}{rcll}
(-\Delta)^s u & = & f & {\text{ in }}\O,\\
      \mathcal{N}_{s}u & = & 0 &{\text{ in }} \RR^{d}\setminus \O,
\end{array}\right.
\end{equation}
where $(-\Delta)^s$ is the fractional Laplacian operator defined by
\begin{equation}\label{frac lap}
(-\Delta)^su(x)=a_{d,s}\; \text{P.V.}\int_{\re^d}\frac{u(x)-u(y)}{|x-y|^{d+2s}}\,dy,
\end{equation}
 $a_{d,s}>0$ being  a  normalization constant given by the Fourier transform (see e.g. \cite{ldk}, \cite{lenoperal}, \cite{valpal} and the references
therein for more properties of this operator), and
\begin{equation}\label{Nmn def}
\mathcal{N}_{s}u(x)=a_{d,s}\Int_{\Omega}\frac{u(x)-u(y)}{|x-y|^{d+2s}}dy, \quad x \in \re^{d}\setminus \overline{\Omega}.
\end{equation}
Notice that as a consequence of the analysis of the sequence of eigenvalues with Neumann condition done in \cite{drv}, we reach that the best constant for the
Hardy inequality with Neumann condition is $0$  and it is attained by any constant function.

With this meaning for the  Neumann condition, the authors in \cite{LMPAS} studied the behavior of the eigenvalues for mixed Dirichlet-Neumann  problems in terms
of the boundary conditions. In particular, they proved a necessary  and sufficient condition for the convergence of the first  eigenvalue of mixed problems to $0$, the principal
eigenvalue for Neumann problem.

 These  previous  results are the inspiration for  our main goal in this paper:   study the attainability of the fractional Hardy constant with mixed boundary
 condition.

More precisely, let $\O\subset \re^d$ be a regular bounded domain containing the origin and consider $N$ and $D$ to be two open sets of $\re^{d}\backslash\Omega$
such that
  $$N \cap D=\emptyset\hbox{  and }     \overline{N}\cup \overline{D}= \re^{d}\backslash\Omega.$$
A such pair $(D,N)$ will be called a \textit{Dirichlet-Neumann configutarion}, D-N configuration to be short.

We define
$$
\Lambda_{N}\equiv\Lambda_{N}(\O)=\inf_{\{\phi\in \mathbb{E}^{s}(\Omega, D), \phi\neq
0\}}\dfrac{\dfrac{a_{d,s}}{2}\displaystyle\int\int_{\mathcal{D}_{\Omega}}\dfrac{|\phi(x)-\phi(y)|^{2}}{|x-y|^{d+2s}}dxdy}{\displaystyle\io
\frac{\phi^2}{|x|^{2s}}\,dx}
$$
where
$$\mathbb{E}^{s}(\Omega,D)=\left\{ u \in H^s(\RR^d):\, u=0 \text{ in } D\right\}.$$
 The above minimizing problem is strongly related to the next \textit{eigenvalue problem}
\begin{equation}\label{P}
\left\{
\begin{array}{rcll}
(-\Delta)^s u & = & \L_N \dfrac{u}{|x|^{2s}} & {\text{ in }}\O,\\
   u & > & 0 &{\text{ in }} \Omega, \\
   \mathcal{B}_{s}u & = & 0 &{\text{ in }} \RR^{d}\backslash \O.
\end{array}\right.
\end{equation}
The mixed boundary condition $\mathcal{B}_{s}$ for the D-N configuration  given by $D$ and $N$  open sets in $\RR^{d}\backslash\Omega$ such that $N \cap D=\emptyset$ and $\overline{N}\cup \overline{D}= \re^{d}\setminus\Omega$, is defined by
\begin{equation}\label{CUST}
\mathcal{B}_{s}u=u\chi_{D}+\mathcal{N}_{s}u\chi_{N},
\end{equation}
and $\mathcal{N}_{s}$ is defined in \eqref{Nmn def}.
As
customary, in \eqref{CUST}, we denoted by $\chi_{A}$ the characteristic function of a set $A$.

In the local case $s=1$, we can mention the works \cite{ACP1} and \cite{ACP2} where the authors have found some conditions  of monotonicity that ensure the attainability or not of the Hardy constant. As a consequence they analyze a doubly-critical problem related to a mixed Sobolev constant.  In \cite{ACP2} the authors deal with the same type of problem associated  to the divergence form elliptic operators associated to  the Caffarelli-Khon-Niremberg
inequalities.

We will extend the previous results to the fractional laplacian framework without any condition of monotonicity and then we get a stronger results than in the local case.

A different kind of nonlocal mixed boundary conditions for the fractional Laplacian was defined recently by several authors, see for example \cite{CT}, where the
authors used the Caffarelli-Silvestre extension to define a suitable nonlocal Neumann boundary condition. We refer also to \cite{gru} for other type of Neumann
condition.

Notice that in the local case one of the main tools to analyze the compactness of the minimizing sequence is to use a suitable Concentration-Compactness
argument that allows to avoid any concentration in $\Omega$ or at the boundary of $\Omega$. In the nonlocal setting we will  consider an alternative approach.  Other
point that gives the difference between the local and the nonlocal case is the fact that in the nonlocal case, the set $N$ can be unbounded and little is known about the regularity
of the solution  in this set.

There are a large literature about semilinear perturbations of the fractional laplacian with Dirichlet boundary conditions. Notice that the  subcritical concave-convex
problem with mixed boundary conditions is studied in \cite{ADV}.

The paper is organized as follows. In Section \ref{prelim}, we introduce some analytical tools needed to study the problem ($P_{\lambda}$), as the natural
fractional Sobolev space associated to problem $(P_\l)$, including some classical functional inequalities and the adaptation of a  Picone inequality type obtained
in \cite{lenoperal}.

The  Hardy constant for mixed problems is treated in Section \ref{sec:Hardy}. We begin by proving a necessary and sufficient condition  for the attainability of the mixed Hardy
constant. The proof is more involved than in the local case and  we prove some previous sharp  estimates to obtain the main result. In Subsection \ref{+} we give
sufficient condition to guarantee the  attainability of $\L_N$; in Subsection \ref{-} we analyze the non-attainability. In both cases we give  explicit examples
where these condition are  realized.

In the last section, Section 4,  among others results, we  study the solvability of the doubly-critical problem. This result is related to the paper \cite{DMPS}
where it is analyzed the  problem in the whole space $\mathbb{R}^d$.

\section{Preliminaries and functional setting.}\label{prelim}
 We introduce in  this section the natural functional framework for our problem and we give some properties and some  embedding results needed when we deal with problem ($P_{\l}$).

Notice that in the whole paper and for simplicity of tipping, we set
\begin{equation}\label{kernel}
d\nu=\dfrac{dx\,dy}{|x-y|^{d+2s}}.
\end{equation}
According to the definition of the fractional Laplacian, see \cite{valpal}, and the integration by parts formula, see \cite{drv}, it is natural to introduce the
following spaces. We denote by $H^s(\RR^d)$ the classical fractional Sobolev space,
\begin{equation}\label{}
H^s(\RR^d)=\left\{u \in L^2(\RR^d):\frac{|u(x)-u(y)|}{|x-y|^{\frac{d}{2}+s}} \in L^2(\RR^d\times \RR^d)\right\},
\end{equation}
endowed with the norm
\begin{equation}\label{}
    \|u\|^2_{H^s(\RR^d)}=
\|u\|_{L^2(\RR^d)}^2 + \dfrac{a_{d,s}}{2}\int\int_{\RR^d\times \RR^d}|u(x)-u(y)|^{2}\,d\nu,
\end{equation}

It is clear that  $H^s(\RR^d)$ is a Hilbert space.

We recall now the classical Sobolev inequality that is proved for instance in \cite{valpal}. See also \cite{Ponce} for an elegant and elementary proof.
\begin{Proposition}\label{PSO}
 Let $s\in (0,1)$ with $d>2s$. There exists a positive constant $S=S(d,s)$ such that,
for any function $u \in H^{s}(\RR^d)$, we have
\begin{equation}\label{sobo00}
    S \|u\|_{L^{2^{*}_{s}}(\RR^d)}^2 \leq \dfrac{a_{d,s}}{2}\Int\Int_{\RR^{d}\times\RR^{d}}|u(x)-u(y)|^{2}\,d\nu,
  \end{equation}
where $2^{*}_{s}=\frac{2d}{d-2s}$.
\end{Proposition}
Beside to the Hardy inequalities \eqref{Hardyint} and  \eqref{hardy}, in  the case of bounded domain $\O$, we have the next \textit{ regional }version of the
Hardy inequality which proof can be found in \cite{ABE}.
\begin{Proposition}\label{cor}
Let $\Omega$ be a bounded regular domain such that $0\in \Omega$, then there exists a constant $C\equiv C(\Omega,s,d)>0$ such that for all $u\in
\mathcal{C}^\infty_0(\Omega)$, we have
\begin{equation}\label{boun}
\begin{array}{rcl}
C\dint\limits_{\Omega} \dfrac{|u(x)|^2}{|x|^{2s}}\,dx &\leq&\dint\limits_{\Omega}\dint\limits_{\Omega}(u(x)-u(y))^2\,d\nu.
\end{array}
\end{equation}
\end{Proposition}
Since we are considering a problem with mixed boundary condition we need to specify the space where the solutions belong.
\begin{Definition}\label{def1}
Let $\O$ be a bounded domain of $\RR^d$ and $D\subset\mathbb{R}^d\setminus\Omega$ an open set. For $0<s<1$, we define the space
$$\mathbb{E}^{s}(\Omega,D)=\left\{ u \in H^s(\RR^d):\, u=0 \text{ in } D\right\}$$
which  is a Hilbert space endowed with the norm induced by $H^s(\RR^d)$.
\end{Definition}
For $u\in\mathbb{E}^{s}(\Omega,D)$, we set

$$\|u\|^{2}=\dfrac{a_{d,s}}{2}\int\int_{\mathcal{D}_{\Omega}}|u(x)-u(y)|^{2}\,d\nu.$$
The properties of this norm are described by the following result. We refer to \cite{drv}, \cite{BaM} and \cite{LMPAS} for the proof and other properties of this
space.
\begin{Proposition}\label{norm eqv}
The norm $\|\,.\, \|$ in $\mathbb{E}^{s}(\Omega,D)$ is equivalent to the one induced by $H^s(\re^d)$, and then $(\mathbb{E}^{s}(\O,D), \langle\, ,\,\rangle)$ is a
Hilbert space with scalar product given by
$$\langle u,v\rangle=\dfrac{a_{d,s}}{2}\int\int\limits_{\mathcal{D}_{\Omega}}(u(x)-u(y))(v(x)-v(y))\,d\nu, $$ moreover there exists a positive constant $C(\O)$
such that  the next Poincaré inequality holds:
\begin{equation}\label{poincare}
C(\O)\int_{\Omega} u^2(x) dx \le \int\int\limits_{\mathcal{D}_{\Omega}}(u(x)-u(y))^2\,d\nu \mbox{  for all  }u\in \mathbb{E}^{s}(\Omega,D).
\end{equation}
\end{Proposition}

 As a consequence of the definition of $\mathbb{E}^{s}(\Omega,D)$ and using the extension result proved in  \cite{drv}, we get the next Sobolev inequality in the
 space $\mathbb{E}^{s}(\Omega,D)$.
\begin{Proposition}\label{sobolev} Let $(D,N)$ be a D-N configuration.
Suppose that $s\in (0,1)$ and $d>2s$. There exists a positive constant $S(N)>0$ such that, for all $u \in \mathbb{E}^{s}(\Omega,D)$, we have
$$
S(N)\|u\|^2_{L^{2^*_s}(\O)}\leq \|u\|^{2}.
$$
\end{Proposition}

The following result justifies the choice of nonlocal \textit{boundary} condition.
\begin{Proposition}\label{BYP} Let $(D,N)$ be a D-N configuration and  $s \in (0,1)$ then for all $u,\,v \in \mathbb{E}^{s}(\Omega,D)$ we have,
\begin{equation}\label{bpa}\int_{\O}v(-\Delta)^su\,dx=
\frac{a_{d,s}}{2}\int\int\limits_{\mathcal{D}_{\Omega}}
     (u(x)-u(y))(v(x)-v(y))\,d\nu -\Int_{N}v\mathcal{N}_su\,dx.
     \end{equation}
\end{Proposition}
 The proof of this result follows by the application of the integration by parts formula given in  Lemma 3.3 of \cite{drv}.

In order to obtain some a priori estimates, we will use the next Picone type inequality that is a extension of the corresponding inequality in $H^s_0(\O)$ obtained
in \cite{lenoperal}. For the reader convenience we give the proof.
\begin{Theorem}\label{piconeth}Let $(D,N)$ be a D-N configuration.
Consider $u,\, v \in \mathbb{E}^{s}(\Omega,D)$ with $u\ge 0$ in $\re^d$ and $u>0$ in $\O\cup N$. Assume that $(-\Delta)^{s}u\geq 0$ is a bounded Radon measure in
$\O$, then
\begin{equation}\label{picone0}
\int_{N}\frac{|v|^2}{u}\mathcal{N}_su\,dx+\int_{\O}\frac{|v|^2}{u}(-\Delta)^{s}u\, dx \leq \dfrac{a_{d,s}}{2}\int\int_{\mathcal{D}_{\O}}(v(x)-v(y))^{2}\,d\nu.
\end{equation}
In particular, if we have equality in \eqref{picone0}, then there exists a constant $C$ such that $v=C u$ in $\re^d$.
\end{Theorem}
\begin{proof}
Notice that for $u, v$ as in the pointwise hypotheses of the Theorem we have the following simple identity
$$
(v(x)-v(y))^{2}-\big(\frac{v^2(x)}{u(x)}-\frac{v^2(y)}{u(y)}\big)\big(u(x)-u(y)\big)=
\bigg(v(x)\bigg(\frac{u(y)}{u(x)}\bigg)^{\frac12}-v(y)\bigg(\frac{u(x)}{u(y)}\bigg)^{\frac12}\bigg)^2.
$$
Integrating the previous identity respect to  $d\nu$, defined in\eqref{kernel},  we conclude.

Finally, if we have the  equality in \eqref{picone0}, then we conclude that
$$
\int\int_{\mathcal{D}_{\O}}\bigg(v(x)\bigg(\frac{u(y)}{u(x)}\bigg)^{\frac12}-v(y)\bigg(\frac{u(x)}{u(y)}\bigg)^{\frac12}\bigg)^2d\nu=0
$$
Thus $v(x)\bigg(\dfrac{u(y)}{u(x)}\bigg)^{\frac12}=v(y)\bigg(\dfrac{u(x)}{u(y)}\bigg)^{\frac 12}$ for almost all $(x,y)\in \mathcal{D}_{\O}$. In particular, if $v(y_0)\neq 0$ for some $y_0\in \O\cup
N$, then  $\dfrac{v(x)}{v(y_0)}=\dfrac{u(x)}{u(y_0)}$. Thus $v(x)=\dfrac{v(y_0)}{u(y_0)}u(x)$ and the result follows.
\end{proof}

\section{Analysis of the mixed Hardy optimal constant}\label{sec:Hardy}
Consider $\Omega\subset \mathbb{R}^d$ a bounded domain and $D, N\subset\mathbb{R}^d\setminus\Omega$  a D-N configuration. In this section we will analyze the condition for
the attainability of the \textit{mixed Hardy constant} defined by
\begin{equation}\label{hardymix}
\Lambda_{N}\equiv\Lambda_{N}(\O)=\inf_{\{\phi\in \mathbb{E}^{s}(\Omega, D), \phi\neq
0\}}\dfrac{\dfrac{a_{d,s}}{2}\displaystyle\int\int_{\mathcal{D}_{\Omega}}|\phi(x)-\phi(y)|^{2}\,d\nu}{\displaystyle\io \frac{\phi^2}{|x|^{2s}}\,dx}.
\end{equation}
We start by proving the following result.
\begin{Theorem}
Assume that $\O\subset\re^d$ is a smooth bounded domain with $0\in\O$. Let $N$, $D$ be two nonempty open sets of $\re^{d}\backslash\Omega$ such that $N \cap
D=\emptyset$ and $\overline{N}\cup \overline{D}= \re^{d}\backslash\Omega$, then
$$0<\Lambda_{N}\le\Lambda.$$
\end{Theorem}

\begin{proof}
We begin by proving the positivity of $\L_N$. Let $u\in \mathbb{E}^{s}(\Omega, D)$ and fix $\d>0$ such that $B_{2\d}(0)\subset \Omega$. Consider $\varphi\in
\mathcal{C}^\infty_0(\O)$ be such that $0\le \varphi\le 1$, $\varphi\equiv 1$ in $B_{\d}(0)$ and $\varphi=0$ in $\O\backslash B_{2\d}(0)$. In what follows we
denote by $C$ or $C(\O)$ any positive constant that depends on $\O,d,s$, that is independent of $u$ and that can be change from line to line.

It is clear that $u=\varphi u+(1-\varphi)u$, thus
\begin{equation}\label{estim1}
\dyle\io \frac{u^2}{|x|^{2s}}\,dx= \dyle\io \frac{(u\varphi)^2}{|x|^{2s}}\,dx+ \io \frac{u^2(1-\varphi)^2}{|x|^{2s}}\,dx+2\dyle \io
\frac{u^2\varphi(1-\varphi)}{|x|^{2s}}\,dx.
\end{equation}
Since $1-\varphi=0$ in $B_\d(0)$, then using the Poincaré inequality we conclude that
\begin{equation}\label{estim2}
\io \frac{u^2(1-\varphi)^2}{|x|^{2s}}\,dx+2\dyle \io \frac{u^2\varphi(1-\varphi)}{|x|^{2s}}\,dx\le
C(\O)\displaystyle\int\int_{\mathcal{D}_{\Omega}}(u(x)-u(y))^{2}\,d\nu.
\end{equation}
We deal now with the term $\dyle\io \dfrac{(u\varphi)^2}{|x|^{2s}}\,dx$. Since $u\varphi\in H^s_0(\O)$, then by the Hardy inequality  \eqref{hardy}, we obtain that
\begin{equation}\label{estim3}
\io \frac{(u\varphi)^2}{|x|^{2s}}\,dx\le C(\O)\displaystyle\int_\Omega\int_{\Omega}\bigg((u\varphi)(x)-(u\varphi)(y)\bigg)^{2}\,d\nu.
\end{equation}
The immediate algebraic identity
$$
\begin{array}{lll}
& \bigg(u(x)\varphi(x)-u(y)\varphi(y)\bigg)^2=\bigg(u(x)-u(y)\bigg)^2\varphi^2(x)+u^2(y)\bigg(\varphi(x)-\varphi(y)\bigg)^2\\ &
+2u(y)\varphi(x)\bigg(u(x)-u(y)\bigg)\bigg(\varphi(x)-\varphi(y)\bigg),
\end{array}
$$
 implies that
$$
\begin{array}{lll}
&\dyle\int_\Omega\int_{\Omega}(u(x)\varphi(x)-u(y)\varphi(y))^{2}\,d\nu=\\ \\ & \dyle
\int_\Omega\int_{\Omega}(u(x)-u(y))^2\varphi^2(x)\,d\nu+\int_\Omega\int_{\Omega} u^2(y)(\varphi(x)-\varphi(y))^2\,d\nu\\ \\ &+2\dyle \int_\Omega\int_{\Omega}
u(y)\varphi(x)(u(x)-u(y))(\varphi(x)-\varphi(y))\,d\nu\\ \\ &=J_1+J_2+2J_3.
\end{array}
$$
In first place, it is clear that
$$
J_1\le C(\O)\int\int_{\mathcal{D}_{\O}}\big(u(x)-u(y)\big)^{2}\,d\nu.
$$
Respect to $J_2$, since $\Omega$ is a bounded domain, it holds that
$$J_2\le C(\O)\int_{\O}\int_\O \frac{u^2(y)}{|x-y|^{d+2s-2}}dxdy \le C(\O)\int_{\O}u^2(y) \,dy\int_{|\xi|\le R}\frac{1}{|\xi|^{d+2s-2}}\,d\xi.$$
Since $\dint_{|\xi|\le\e}\frac{1}{|\xi|^{d+2s-2}}\,d\xi<\infty$, using the Poincaré inequality in Proposition \ref{norm eqv}, we get
$$
J_2\le C(\O)\int_{\O}\int_\O u^2(y)dy \le C(\O)\int\int_{\mathcal{D}_{\O}}\big(u(x)-u(y)\big)^{2}\,d\nu.
$$
By using Young inequality, we reach that
$$
|J_3|\le C_1 J_1+C_2J_2\le C(\O)\int\int_{\mathcal{D}_{\O}}\big(u(x)-u(y)\big)^{2}\,d\nu.
$$
Therefore, combining the above inequalities, we conclude that
$$
\dyle\int_\Omega\int_{\Omega}(u(x)\varphi(x)-u(y)\varphi(y))^{2}\,d\nu \le C(\O)\int\int_{\mathcal{D}_{\O}}\big(u(x)-u(y)\big)^{2}\,d\nu.
$$
Going back to \eqref{estim1},  and by using the estimates \eqref{estim2}, \eqref{estim3}, we conclude that
$$
\dyle\io \frac{u^2}{|x|^{2s}}\,dx \le C(\O)\int\int_{\mathcal{D}_{\O}}\big(u(x)-u(y)\big)^{2}\,d\nu.
$$
Hence $\L_N>0$ and then the first affirmation follows.

 Finally, since $H^s_0(\O)\subset \mathbb{E}^{s}(\Omega, D)$,  by definition it follows  that $\Lambda_{N}\le\Lambda$.
\end{proof}

The main result  in this section is the following.
\begin{Theorem}\label{exist00}
Assume that $\O\subset\re^d$ is a smooth bounded domain with $0\in\O$. Let $N$, $D$ be two open sets of $\re^{d}\backslash\Omega$ such that $N \cap D=\emptyset$
and $\overline{N}\cup \overline{D}= \re^{d}\backslash\Omega$, then $\Lambda_{N}<\Lambda$ if and only if $\Lambda_N$ is attained.
\end{Theorem}

We split the proof into two parts contained in the following two  subsections.

\subsection{If $\Lambda_{N}<\Lambda$, then $\Lambda_N$ is attained}\label{+}

\begin{Proposition}\label{exist01}
In  the hypotheses of Theorem \ref{exist00}, assume that  $\Lambda_{N}<\Lambda$. Then, $\Lambda_N$ is attained
\end{Proposition}
\begin{proof}
Let $\{u_n\}_n\subset \mathbb{E}^{s}(\Omega, D)$ be a minimizing sequence  for $\L_N$ defined in \eqref{hardymix} with $\dyle \int_{\O}\dfrac{u_n^2}{|x|^{2s}}\,
dx=1$, then $\{u_n\}_n$ is bounded in $\mathbb{E}^{s}(\Omega, D)$, and
$$\dfrac{a_{d,s}}{2}\Int\Int_{\RR^{d}\times\RR^{d}}|u_n(x)-u_n(y)|^{2}\,d\nu\to \Lambda_{N}.$$ Without loss of generality we can assume that $u_n\ge 0$ for all
$n$.
Hence we get the existence of $\bar{u}\in \mathbb{E}^{s}(\Omega, D)$ such that $u_n \rightharpoonup \bar{u}$ weakly in $\mathbb{E}^{s}(\Omega, D)$, and up to a
subsequence, $u_n\to \bar{u}$ strongly in $L^\s(\O)$ for all $\s<2^*_s$ and $u_n\to \bar{u}$ a.e in $\O$.

We claim that $\bar{u}\neq 0$. We argue by contradiction. Assume that $\bar{u}=0$ and let $R>0$ be such that $B_{4R}(0)\subset\O $. Consider $\varphi\in
\mathcal{C}^\infty_0(\O)$ such that $0\le \varphi\le 1$, $\varphi\equiv 1$ in $B_R(0)$ and $\varphi=0$ in $\RR^{d}\backslash B_{2R}(0)$, we define $w_n=\varphi
u_n$, then $w_n\in H^s_0(\O)$ and
\begin{equation}\label{madrid}
\Lambda\le\dfrac{\dfrac{a_{d,s}}{2}\displaystyle\int\int_{\mathcal{D}_{\Omega}}|w_n(x)-w_n(y)|^{2}\,d\nu}{\displaystyle\io \frac{w_n^2}{|x|^{2s}}\,dx},
\end{equation}
Notice that
\begin{eqnarray*}
\dyle\io \frac{w_n^2}{|x|^{2s}}\,dx & = & \dyle\io \frac{u_n^2\varphi^2}{|x|^{2s}}\,dx=\dyle\io \frac{u_n^2}{|x|^{2s}}\,dx+\dyle\io
\frac{u_n^2(\varphi^2-1)}{|x|^{2s}}\,dx\\ &=& \dyle\io \frac{u_n^2}{|x|^{2s}}\,dx+\dyle\int_{B_R(0)}
\frac{u_n^2(\varphi^2-1)}{|x|^{2s}}\,dx+\dyle\int_{\O\backslash B_R(0)} \frac{u_n^2(\varphi^2-1)}{|x|^{2s}}\,dx\\ &=& 1+\dyle\int_{\O\backslash B_R(0)}
\frac{u_n^2(\varphi^2-1)}{|x|^{2s}}\,dx\\ &=& 1+o(1).
\end{eqnarray*}
We have
$$\int\int_{\mathcal{D}_{\Omega}}|w_n(x)-w_n(y)|^{2}\,d\nu=\int\int_{\mathcal{D}_{\Omega}}|u_n(x)\varphi(x)-u_n(y)\varphi(y)|^{2}\,d\nu.$$
Since
$$
\begin{array}{lll}
(u_n(x)\varphi(x)-u_n(y)\varphi(y))^2 &= & \big((u_n(x)-u_n(y))\varphi(x)+u_n(y)(\varphi(x)-\varphi(y))\big)^2\\ \\ &= &
(u_n(x)-u_n(y))^2\varphi^2(x)+u_n^2(y)(\varphi(x)-\varphi(y))^2\\ \\ &+ & 2u_n(y)\varphi(x)(u_n(x)-u_n(y))(\varphi(x)-\varphi(y)),
\end{array}
$$
it holds that
$$
\begin{array}{lll}
\dyle \int\int_{\mathcal{D}_{\Omega}}|w_n(x)-w_n(y)|^{2}\,d\nu & = & \dyle \int\int_{\mathcal{D}_{\Omega}}|u_n(x)\varphi(x)-u_n(y)\varphi(y)|^{2}\,d\nu\\ &=&
\dyle \int\int_{\mathcal{D}_{\Omega}}(u_n(x)-u_n(y))^2\varphi^2(x)\,d\nu+\int\int_{\mathcal{D}_{\Omega}}u_n^2(y)(\varphi(x)-\varphi(y))^2\,d\nu\\ &+& 2\dyle
\int\int_{\mathcal{D}_{\Omega}}u_n(y)\varphi(x)(u_n(x)-u_n(y))(\varphi(x)-\varphi(y))\,d\nu\\ &=& I_1(n)+I_2(n)+2I_3(n).
\end{array}
$$
Let us begin by estimating the term $I_2(n)$. Recall that $d\nu=\dfrac{dx\,dy}{|x-y|^{d+2s}}$, then we have
$$
\begin{array}{lll}
I_2(n) & = & \dyle \int\int_{\mathcal{D}_{\Omega}}(\varphi(x)-\varphi(y))^2u_n^2(y)\,d\nu\\ \\ &= &
\dyle\int_{\Omega}\int_{\Omega}(\varphi(x)-\varphi(y))^2u_n^2(y)\,d\nu+ \int_{\Omega}\int_{\Omega^c}(\varphi(x)-\varphi(y))^2u_n^2(y)\,d\nu\\  \\ &+&   \dyle
\int_{\Omega^c}\int_{\Omega}(\varphi(x)-\varphi(y))^2u_n^2(y)\,d\nu \\ \\ &= & I_2^1(n)+I_2^2(n)+I_2^3(n).
\end{array}
$$
Taking into account that for all $(x,y)\in \O\times \O$, $(\varphi(x)-\varphi(y))^2\le C(\O)|x-y|^2$, we reach that
$$I_2^1(n)\le C(\O)\int_{\O}\int_\O \frac{u_n^2(y)}{|x-y|^{d+2s-2}}dxdy \le C(\O)\int_{\O}u_n^2(y) \,dy\int_{|\xi|\le\e}\frac{1}{|\xi|^{d+2s-2}}\,d\xi.$$
Since $\dint_{|\xi|\le\e}\frac{1}{|\xi|^{d+2s-2}}\,d\xi<\infty$, then $I_2^1(n)= o(1).$

Next we proceed to estimate  $I_2^2(n)$. We have
\begin{eqnarray*}
I_2^2(n) & = & \dyle \int_{\Omega}\int_{\Omega^c}(\varphi(x)-\varphi(y))^2u_n^2(y)\,d\nu=\int_{\Omega}\int_{\Omega^c} \varphi^2(y)u_n^2(y)\,d\nu\\ &=& \dyle
\int_{\Omega}\varphi^2(y)u_n^2(y)\int_{\Omega^c}\frac{1}{|x-y|^{d+2s}}\,dx\,dy=\int_{|y|\le 2R}\varphi^2(y)u_n^2(y)\int_{|x|\ge 4R}\frac{1}{|x|^{d+2s}}\,dx\,dy\\
&=& o(1).
\end{eqnarray*}

Finally, we consider the term $I_3(n)$.
\begin{eqnarray*}
I_3 (n) &&=\int\int_{\mathcal{D}_{\Omega}}u_n(y)\varphi(x)(u_n(x)-u_n(y))(\varphi(x)-\varphi(y))\,d\nu\\
&&=\int_{\Omega}\int_{\Omega}u_n(y)\varphi(x)(u_n(x)-u_n(y))(\varphi(x)-\varphi(y))d\nu\\
&&+\int_{\Omega}\int_{\Omega^c}u_n(y)\varphi(x)(u_n(x)-u_n(y))(\varphi(x)-\varphi(y))\,d\nu\\ &&+\int_{\Omega^c}\int_{\Omega}
u_n(y)\varphi(x)(u_n(x)-u_n(y))(\varphi(x)-\varphi(y))\,d\nu\\ &&=I_3^1(n)+I_3^2(n)+I_3^3(n).
\end{eqnarray*}
We start by the first term. Using H\"older inequality, we get
$$
\begin{array}{lll}
I_3^1(n) &=& \dyle \int_{\Omega}\int_{\Omega} u_n(y)\varphi(x)(u_n(x)-u_n(y))(\varphi(x)-\varphi(y))\,d\nu\\ &\le & \dyle \Big(\int_{\Omega}\int_{\Omega}
(u_n(x)-u_n(y))^2\varphi^2(x)\,d\nu\Big)^{\frac 12}\Big(\int_{\Omega}\int_{\Omega} (\varphi(y)-\varphi(x))^2u_n^2(y)\,d\nu\Big)^{\frac 12}\\ \\ &\le & C
I_2^1(n)=o(1).
\end{array}
$$
Now, since $\varphi(x)=0$ if $x\in \O^c$, then $I_3^2(n)=0$.

Combing the above estimates, we conclude that
\begin{equation}\label{main00}
\dyle \int\int_{\mathcal{D}_{\Omega}}|w_n(x)-w_n(y)|^{2}\,d\nu = \dyle \int\int_{\mathcal{D}_{\Omega}}(u_n(x)-u_n(y))^2\varphi^2(x)\,d\nu+ I_2^3(n)+2I_3^3(n)
+o(1).
\end{equation}
Notice that
\begin{eqnarray*}
I_2^3(n)+2I_3^3(n) &&=\int_{N}\int_{\O}(\varphi(x)-\varphi(y))^2u_n^2(y)\,d\nu\\
&&+2\int_{N}\int_{\O}u_n(y)\varphi(x)(u_n(x)-u_n(y))(\varphi(x)-\varphi(y))\,d\nu\\ &&=\int_{N}\int_{\O}\varphi^2(x)u_n^2(y) \,d\nu
+2\int_{N}\int_{\O}u_n(y)\varphi^2(x)(u_n(x)-u_n(y))\,d\nu\\ &&=\int_{N}\int_{\O}\varphi^2(x)u_n^2(y) \,d\nu
+\int_{N}\int_{\O}2u_n(y)u_n(x)\varphi^2(x)\,d\nu-2\int_{N}\int_{\O}u_n^2(y)\varphi^2(x)\,d\nu\\ &&=-\int_{N}\int_{\O}\varphi^2(x)u_n^2(y)\,d\nu
+\int_{N}\int_{\O}2u_n(y)u_n(x)\varphi^2(x)\,d\nu.
\end{eqnarray*}
Using now Young inequality, we get that
$$
\begin{array}{lll}
I_2^3(n)+2I_3^3(n) & \le & \dyle -\int_{N}\int_{\O}\varphi^2(x)u_n^2(y)\,d\nu+\e\int_{N}\int_{\O} u_n^2(y)\varphi^2(x)\,d\nu+C_\e
\int_{N}\int_{\O}u_n^2(x)\varphi^2(x)\,d\nu\\ &\le & \dyle  (\e-1)\int_{N}\int_{\O}\varphi^2(x)u_n^2(y)\,d\nu+C_\e \int_{N}\int_{\O} u_n^2(x)\varphi^2(x)\,d\nu.
\end{array}
$$
Choosing $\e$ small enough, we obtain that
$$
\begin{array}{lll}
I_2^3(n)+2I_3^3(n) & \le & C_\e\dyle \int_{\re^d\backslash B_{4R}(0)}\int_{B_{2R}(0)} u_n^2(x)\varphi^2(x)\,d\nu\\  \\ &\le & C(R, \e)\dyle
\int_{B_{2R}(0)}u_n^2(x)\,dx= o(1).
\end{array}
$$
Therefore, from \eqref{main00}, it holds that
\begin{equation}\label{main000}
\dyle \int\int_{\mathcal{D}_{\Omega}}|w_n(x)-w_n(y)|^{2}\,d\nu\le \dyle \int\int_{\mathcal{D}_{\Omega}}(u_n(x)-u_n(y))^2\varphi^2(x)\,d\nu +o(1).
\end{equation}
Going back to \eqref{madrid}, we conclude that
$$
\begin{array}{lll}
\Lambda_N<\Lambda &\le & \dfrac{\dfrac{a_{d,s}}{2}\displaystyle\int\int_{\mathcal{D}_{\Omega}}|w_n(x)-w_n(y)|^{2}\,d\nu}{\displaystyle\io
\frac{w_n^2}{|x|^{2s}}\,dx}\le \dfrac{\dfrac{a_{d,s}}{2}\displaystyle\int\int_{\mathcal{D}_{\Omega}}|u_n(x)-u_n(y)|^{2}\,d\nu}{1+o(1)}\\ & =& \L_N+o(1),
\end{array}
$$
which is a contradiction with the hypothesis $\L_N<\L$. Hence $\bar{u}\neq 0$ and then the claim follows.

To show that $\L_N$ is achieved we will use the Ekeland variational principle, \cite{Ek}, then up to a subsequence, it holds that
\begin{equation}\label{Pnn}
\left\{
\begin{array}{rcll}
(-\Delta)^s u_n & = & \Lambda_{N} \dfrac{u_n}{|x|^{2s}}+o(1) & {\text{ in }}\O,\\
   \mathcal{B}_{s}u_n & = & 0 &{\text{ in }} \RR^{d}\backslash \O,
\end{array}\right.
\end{equation}
Let $\varphi\in \mathbb{E}^{s}(\Omega, D)$, by duality argument we obtain that
$$\int_{\O} (-\Delta)^s u_n\varphi\to \int_{\O} (-\Delta)^s \bar{u}\varphi \mbox{  and  }\int_{\O}\dfrac{u_n\varphi}{|x|^{2s}}\to
\int_{\O}\dfrac{\bar{u}\varphi}{|x|^{2s}} \qquad\text{ as } n \to \infty.$$
Thus $\bar{u}$ solves the problem
\begin{equation}\label{otov}
\left\{
\begin{array}{rcll}
(-\Delta)^s \bar{u} & = & \Lambda_{N} \dfrac{\bar{u}}{|x|^{2s}} & {\text{ in }}\O,\\
   \bar{u}\in \mathbb{E}^{s}(\Omega, D), \bar{u} & \ge & 0 &{\text{ in }} \Omega, \\
   \mathcal{B}_{s}\bar{u} & = & 0 &{\text{ in }} \RR^{d}\backslash \O.
\end{array}\right.
\end{equation}
Choosing $\bar{u}$ as a test function in \eqref{otov}, we obtain
$$
\Lambda_{N}=\dfrac{\dfrac{a_{d,s}}{2}\displaystyle\int\int_{\mathcal{D}_{\Omega}}|\bar{u}(x)-\bar{u}(y)|^{2}\,d\nu}{\displaystyle\io
\frac{\bar{u}^2}{|x|^{2s}}\,dx},
$$
and the result follows.
\end{proof}

\subsubsection{ Properties of the spectral value $\Lambda_N$ if $\L_N<\L$}\label{sec:attained}

In this subsection we treat the case $\L_N<\L$, thus $\L_N$ is achieved and we prove that it behaves like a principal eigenvalue of the mixed elliptic problem
with the Hardy weight.

To start, we begin by giving some  configurations of $(D,N)$  for which the constant $\L_N$ is reached. By the previous results it suffices to prove that
$\L_N<\L$. This last inequality is a straightforward consequence of some results contained in \cite{LMPAS}. For the reader convenience  we below explain some details.

We say that $\O$ is an \textit{{admissible}} domain  if it is a $C^{1,1}$ and it satisfies the exterior sphere condition. Now, let consider sequences of sets
$\{D_k\}_{k\in \mathbb{N}}$, $\{N_k\}_{k\in \mathbb{N}}$ such that $N_k \cap D_k=\emptyset$ and $\overline{N_k}\cup \overline{D_k}= \re^{d}\backslash\Omega$.
Following closely the same argument as in \cite{LMPAS}, we obtain the next result.
\begin{Theorem}\label{infty}
Let $\O$ be an admissible domain and $0<s<\frac 12$. Suppose that  for all $R>0$ we have $\lim\limits_{k\to\infty}|D_k\cap B_R|=0$, then
$\lim\limits_{k\to\infty}\L_k=0$ and as a consequence there exists $k_0\in \ene$ such that $\L_k$ is attained for all $k\ge k_0$.
\end{Theorem}
\begin{proof}
For $\rho>0$ fixed, we define
$$\l_{\rho,k}=\inf\limits_{\{u\in E^s(\Omega, D), ||u||\neq 0\}}
\dyle\dfrac{\dyle \dfrac{a_{d,s}}{2}\int\int_{\mathcal{D}_{\Omega}} |u(x)-u(y)|^2\,d\nu}{\dyle\int_{\Omega}\dfrac{u^2(x)}{|x|^{2s}+\rho}dx},$$ it is clear that
$\L_k\le \l_{\rho,k}$ for all $\rho>0$. From \cite{LMPAS} we know that $\lim\limits_{k\to\infty}\l_{\rho,k}=0$, thus we conclude.
\end{proof}

In the case $\frac 12\le s<1$, we have the following result.
\begin{Theorem}
Assume that $s\in [\frac 12,1)$ and that the hypotheses of Theorem \ref{infty} hold. Assume in addition that for some $\d>0$, we have $dist(D_k,\O)>\d,\: \forall
k\ge k_0$, then $\lim\limits_{k\to\infty}\L_k=0$.
\end{Theorem}

Let us denote by $\bar{u}\in \mathbb{E}^{s}(\Omega, D)$, the function that realize the minimum in \eqref{hardymix} with $\dyle\io \frac{\bar{u}^2}{|x|^{2s}}dx=1$,
then $\bar{u}$ solves the eigenvalue problem
\begin{equation}\label{autoT}
\left\{
\begin{array}{rcll}
(-\Delta)^s u & = & \l \dfrac{u}{|x|^{2s}} & {\text{ in }}\O,\\
   u\in \mathbb{E}^{s}(\Omega, D), u & > & 0 &{\text{ in }} \Omega, \\
   \mathcal{B}_{s} u & = & 0 &{\text{ in }} \RR^{d}\backslash \O.
\end{array}\right.
\end{equation}
with $\l=\L_N$.

In the next result we show the relevant spectral properties of  $\L_N$ when it is reached.
\begin{Theorem}\label{nonauto}
Assume that $\L_N<\L$, then $\L_N$ is an isolated and simple eigenvalue, that is:
\begin{enumerate}
\item If $v$ is an other solution to problem \eqref{autoT} with $\l=\L_N$, then $v=C \bar{u}$,  $C\in \mathbb{R}$. \item There exists $\e>0$ such that for all
    $\l\in (\L_N, \L_N+\e)$, problem \eqref{autoT} has non nontrivial solution.
\end{enumerate}
\end{Theorem}
\begin{proof}
Let us begin by proving the first point. Suppose that $v$ is another solution to problem \eqref{autoT} with $\l=\L_N$. If $v\ge 0$, then using the Picone
alternative in Theorem \ref{piconeth}, we conclude that $v=C\bar{u}$ for some $C\ge 0$.

Now, suppose that $v$ change sign, then using $v_+$ (respectively $v_-$) as a test function in the equation of $v$, we obtain that
$$
\L_N\dyle\int_{\Omega}\dfrac{v_\pm^2(x)}{|x|^{2s}+\e}dx \le \dfrac{a_{d,s}}{2}\dyle\int\int_{\mathcal{D}_{\Omega}} |v\pm(x)-v\pm(y)|^2 \,d\nu.
$$
Hence $v\pm$ realize the minimum in \eqref{hardymix} and then they are nonnegative solutions to problem \eqref{autoT} with $\l=\L_N$. Hence we get the existence
of $C_\pm\ge 0$ such that $v_\pm=C_\pm \bar{u}$, thus $v=v_+-v_-=(C_+-C_-)\bar{u}$.

We prove now that $\L_N$ is isolated. Assume the existence of a sequence $\{(\l_n, u_n)\}_n\subset (\L_N, \infty)\times \mathbb{E}^{s}(\Omega, D)$ such that
$\l_n\downarrow \L_N$ and $u_n$ solves the problem \eqref{autoT} with $\l=\l_n$. Without loss of generality we can assume that $\io \frac{u^2_n}{|x|^{2s}}dx=1$
and $\l_n<\L_0<\L_N$ for all $n$. Thus
$$
\dfrac{a_{d,s}}{2}\int\int_{D_\O}(u_n(x)-u_n(y))^2d\nu=\l_n\le \L.
$$
Hence $\{u_n\}_n$ is bounded in $\mathbb{E}^{s}(\Omega, D)$ and then we get the existence of $\hat{u}\in \mathbb{E}^{s}(\Omega, D)$ such that $u_n\rightharpoonup
\hat{u}$ weakly in $\mathbb{E}^{s}(\Omega, D)$, $u_n\to \hat{u}$ strongly in $L^\s(\O)$ for all $\s<2^*$ and a.e. in $\O\cup D$. It is not difficult to show that
$\hat{u}$ solves problem \eqref{autoT} with $\l=\L_N$. Using the previous step we get the existence of a constant $\hat{C}$ such that $\hat{u}=\hat{C}\bar{u}$.

Now, taking $\bar{u}$ as a test function in the equation of $u_n$ it holds that
\begin{equation}\label{orth}
\io \frac{u_n\bar{u}}{|x|^{2s}}dx=0\mbox{  for all  }n,
\end{equation}
hence
\begin{equation}\label{orth0}
\io \frac{\hat{u}\bar{u}}{|x|^{2s}}dx=0.
\end{equation}
Taking into consideration that $\hat{u}=\hat{C}\bar{u}$, it holds that $\hat{C}=0$ and then $\hat{u}=0$.

Going back to the equation of $u_n$ and by Kato inequality it holds that
$$
(-\Delta)^s |u_n| \le \l_n \dfrac{|u_n|}{|x|^{2s}} \mbox{ in }\O.
$$
Since $\l_n\le\L_0$, then by the result of \cite{AMPA}, we conclude that
\begin{equation}\label{estim-zero}
|u_n(x)|\le C|x|^{-\a} \mbox{  in   }B_\eta(0)\subset\subset \O, \mbox{  for all  }n,
\end{equation}
where $\a_0>0$ and $\L_0$ are related with the next identity
\begin{equation}\label{lambda}
\L_0=\L_0(\alpha_0)=\dfrac{2^{2s}\,\Gamma(\frac{d+2s+2\alpha_0}{4})\Gamma(\frac{d+2s-2\alpha_0}{4})}{\Gamma(\frac{d-2s+2\alpha_0}{4})\Gamma(\frac{d-2s-2\alpha_0}{4})}.
\end{equation}
Since $\L_0<\L$, then $\a_0<\frac{d-2s}{2}$, thus $\dfrac{|u_n|^2}{|x|^{2s}}\le C|x|^{-2\a_0-2s}\in L^1(B_\eta(0))$. By the Dominated Convergence Theorem we reach
that
$$
1=\io \frac{u^2_n}{|x|^{2s}}dx\to \io \frac{\hat{u}^2}{|x|^{2s}}dx \mbox{   as   }n\to \infty,
$$
a contradiction with the fact that $\hat{u}=0$, hence we conclude.
\end{proof}

\subsection{ If $\Lambda_{N}=\Lambda$, then $\Lambda_N$ is not attained}\label{-}
We prove the following result that complete the proof of Theorem \ref{exist00}.
\begin{Proposition} In the hypotheses of Theorem \ref{exist00} assume that $\Lambda_{N}=\Lambda$, then $\Lambda_N$ is not attained.
\end{Proposition}
\begin{proof} We argue by contradiction. Assume that $\L_N$ is attained in $\mathbb{E}^{s}(\Omega, D)$, then there exists $u\in \mathbb{E}^{s}(\Omega, D)$ such
that
\begin{equation*}
\left\{
\begin{array}{rcll}
(-\Delta)^s u & = & \Lambda_{N} \dfrac{u}{|x|^{2s}} & {\text{ in }}\O,\\
   u\in \mathbb{E}^{s}(\Omega, D) , u & > & 0 &{\text{ in }} \Omega, \\
   \mathcal{B}_{s} u & = & 0 &{\text{ in }} \RR^{d}\backslash \O.
\end{array}\right.
\end{equation*}
Let $B_r(0)\subset\subset \O$ and define $v$ to be the unique solution of the problem
\begin{equation*}
\left\{
\begin{array}{rcll}
(-\Delta)^s v & = & 0 & {\text{ in }}\O,\\
   v & = & v_0 &{\text{ in }} \re^{d}\backslash B_r(0),
\end{array}\right.
\end{equation*}
where
$$
v_0(x)= \left\{
\begin{array}{rcll}
u(x) & \mbox{  if  } & x\in \O\backslash B_r(0),\\ 0 & \mbox{  if  } & x\in \re^{d}\backslash \O,
\end{array}\right.
$$
it is clear that $u\ge v$. Setting $w=u-v$ then $w\ge 0$ in $\re^d, w\in H^s(\O)$ and it solves
\begin{equation*}
\left\{
\begin{array}{rcll}
(-\Delta)^s w & = & \Lambda \dfrac{u}{|x|^{2s}}=\Lambda \dfrac{w}{|x|^{2s}}+\Lambda \dfrac{v}{|x|^{2s}} & {\text{ in }}\O,\\
      w & \ge & 0 &{\text{ in }} \RR^{d}\backslash \O.
\end{array}\right.
\end{equation*}
From \cite{AMPA} we know that $w(x)\ge C_1|x|^{-\frac{d-2s}{2}},\  x\in B_{r_0}(0)\subset\subset B_r(0)$, hence
$$\infty=C_1\int_{B_{r_0}(0)} |x|^{-d}\,dx=C_1\int_{B_{r_0}(0)} |x|^{-2^*_s\frac{d-2s}{2}}\,dx \le \int_{B_{r_0}(0)} w^{2^*_s}\,dx$$
which is a contradiction with  the fact that $w\in H^s(\O)$. Hence the result follows.
\end{proof}

\subsubsection{ Examples for which we find $\L_N=\L$}\label{sec:non}
In this subsection we give some geometrical condition to ensure that $\L_N=\L$.
We have the next result.
\begin{Theorem}\label{noexisthardy}
Let $\O\subset\re^d$ be a smooth bounded domain such that $0 \in\O$. Let $w(x)=|x|^{-\frac{d-2s}{2}}$ and suppose that $\mathcal{N}_sw(x)\ge 0$ for all $x\in N$,
then $\L_N=\L$, moreover, the problem
\begin{equation}\label{PSigma1}
\left\{
\begin{array}{rcll}
(-\Delta)^s u & = & \Lambda_{N} \dfrac{u}{|x|^{2s}} & {\text{ in }}\O,\\
   u\in \mathbb{E}^{s}(\Omega, D), u & > & 0 &{\text{ in }} \Omega, \\
   \mathcal{B}_{s} u & = & 0 &{\text{ in }} \RR^{d}\backslash \O,
\end{array}\right.
\end{equation}
 has no solution.
\end{Theorem}
\begin{proof}
We argue by contradiction. Assume that $\mathcal{N}_sw(x)\ge 0$ for all $x\in N$ and that $\L_N<\L$. Then by Theorem \ref{exist00}, we get the existence of
$u_1\in \mathbb{E}^{s}(\Omega, D)$ a positive solution to problem \eqref{PSigma1}. Setting $v_1(x)=\dfrac{u_1(x)}{w(x)}$, then
$$ |v_1(x)-v_1(y)|^2\, w(x) w(y)=|u_1(x)-u_1(y)|^2+u_1^2(x)(\dfrac{w(y)}{w(x)}-1)+u_1^2(y)(\dfrac{w(x)}{w(y)}-1).$$ Thus
\begin{eqnarray*}
&\dyle \dfrac{a_{d,s}}{2}\int\int_{\mathcal{D}_{\Omega}} |v_1(x)-v_1(y)|^2\, w(x) w(y)\,d\nu =\dfrac{a_{d,s}}{2}\int\int_{\mathcal{D}_{\Omega}} |u_1(x)-u_1(y)|^2\,d\nu\\
&\dyle +\dfrac{a_{d,s}}{2}\int\int_{\mathcal{D}_{\Omega}} u^2_1(x)\dfrac{(w(y)-w(x))}{w(x)}\,d\nu+\dfrac{a_{d,s}}{2}\int\int_{\mathcal{D}_{\Omega}} u^2_1(y)\dfrac{(w(x)-w(y))}{w(y)}\,d\nu.
\end{eqnarray*}
According to the symmetry of last two terms of  the above identity, we obtain that
\begin{eqnarray*}
&\dyle \dfrac{a_{d,s}}{2}\int\int_{\mathcal{D}_{\Omega}} |v_1(x)-v_1(y)|^2\, w(x)w(y)\,d\nu =\dfrac{a_{d,s}}{2}\int\int_{\mathcal{D}_{\Omega}} |u_1(x)-u_1(y)|^2\,d\nu\\
&\dyle -\dfrac{a_{d,s}}{2}\int\int_{\mathcal{D}_{\Omega}} \bigg(\frac{u^2_1(x)}{w(x)}-\frac{u^2_1(y)}{w(y)}\bigg)(w(y)-w(x))\,d\nu.
\end{eqnarray*}
Notice that
$$ \frac{a_{d,s}}{2}\int\int\limits_{\mathcal{D}_{\Omega}}
\bigg(\frac{u^2_1(x)}{w(x)}-\frac{u^2_1(y)}{w(y)}\bigg)(w(y)-w(x))\,d\nu=\int_{\O}\frac{u^2_1(x)}{w(x)}(-\Delta)^sw\,dx+ \int_{N}\frac{u^2_1(x)}{w(x)}\mathcal{N}_sw(x)\,dx.$$
Therefore,
$$
\frac{a_{d,s}}{2}\int\int\limits_{\mathcal{D}_{\Omega}}
\bigg(\frac{u^2_1(x)}{w(x)}-\frac{u^2_1(y)}{w(y)}\bigg)(w(x)-w(y))\,d\nu=\L\int_{\Omega}\dfrac{u^2_1(x)}{|x|^{2s}}dx+
\int_{N}\frac{u^2_1(x)}{w(x)}\mathcal{N}_sw(x)\,dx.
$$
Hence, it holds that
\begin{eqnarray*}
&\dyle \dfrac{a_{d,s}}{2}\int\int_{\mathcal{D}_{\Omega}} |v_1(x)-v_1(y)|^2\, w(x) w(y)\,d\nu+\int_{N}\frac{u^2_1(x)}{w(x)}\mathcal{N}_sw(x)\,dx\\ &=\dyle
\dfrac{a_{d,s}}{2}\int\int_{\mathcal{D}_{\Omega}} |u_1(x)-u_1(y)|^2\,d\nu-\L\int_{\Omega}\dfrac{u^2_1(x)}{|x|^{2s}}dx.
\end{eqnarray*}
Since
$$
\dfrac{a_{d,s}}{2}\int\int_{\mathcal{D}_{\Omega}} |u_1(x)-u_1(y)|^2\,d\nu=\L_N\int_{\Omega}\dfrac{u^2_1(x)}{|x|^{2s}}dx,$$ we deduce that
$$
\dyle \dfrac{a_{d,s}}{2}\int\int_{\mathcal{D}_{\Omega}} |v_1(x)-v_1(y)|^2\, w(x) w(y)\,d\nu+\int_{N}\frac{u^2_1(x)}{w(x)}\mathcal{N}_sw(x)\,dx=0.
$$
Thus, if $\mathcal{N}_s(w(x))\ge 0$ for all $x\in N$, it follows that $v_1=0$ and we get a contradiction. Hence we conclude.
\end{proof}
Here we give an explicit bounded domain where the above situation holds.

Define the set $\Omega=\Omega_1\cup \Omega_2\cup \Omega_3$ where
$$
\Omega_1= B_\e(0),\, \Omega_2= \big\{x\in\re^d, \e\le x_1\le A\mbox{  and  }|(x_2,x_3,...,x_d)|<\e\big\},\;\O_3=\big\{x\in\re^d,\, A<|x|<\beta\big\}.
$$
Now, we consider
$$
\begin{array}{lll}
D &= & \bigg\{x\in\re^{d}\backslash \O , \e<|x|<\eta\bigg\}\cup \bigg\{x\in\re^d, \eta\le x_1\le A\mbox{  and  }\e<|(x_2,x_3,...,x_d)|<m\bigg\}\\ &\cup &
\bigg\{x\in \re^d,\,|x|>\beta\bigg\},\end{array}
$$ and
$$ N=\big\{x\in\re^{d}\backslash \{\O\cup D\} , \eta<|x|<A\big\}.
$$
It is clear that $\Omega$ is a bounded domain of $\re^d$,  $N$ and $D$ are two open sets of $\re^{d}\backslash\bar{\Omega}$ with $N \cap D=\emptyset$ and
$\overline{N}\cup \overline{D}= \re^{d}\backslash\Omega$.

\begin{figure}[h!]
\centering
\includegraphics*[width=12cm, height=7cm]{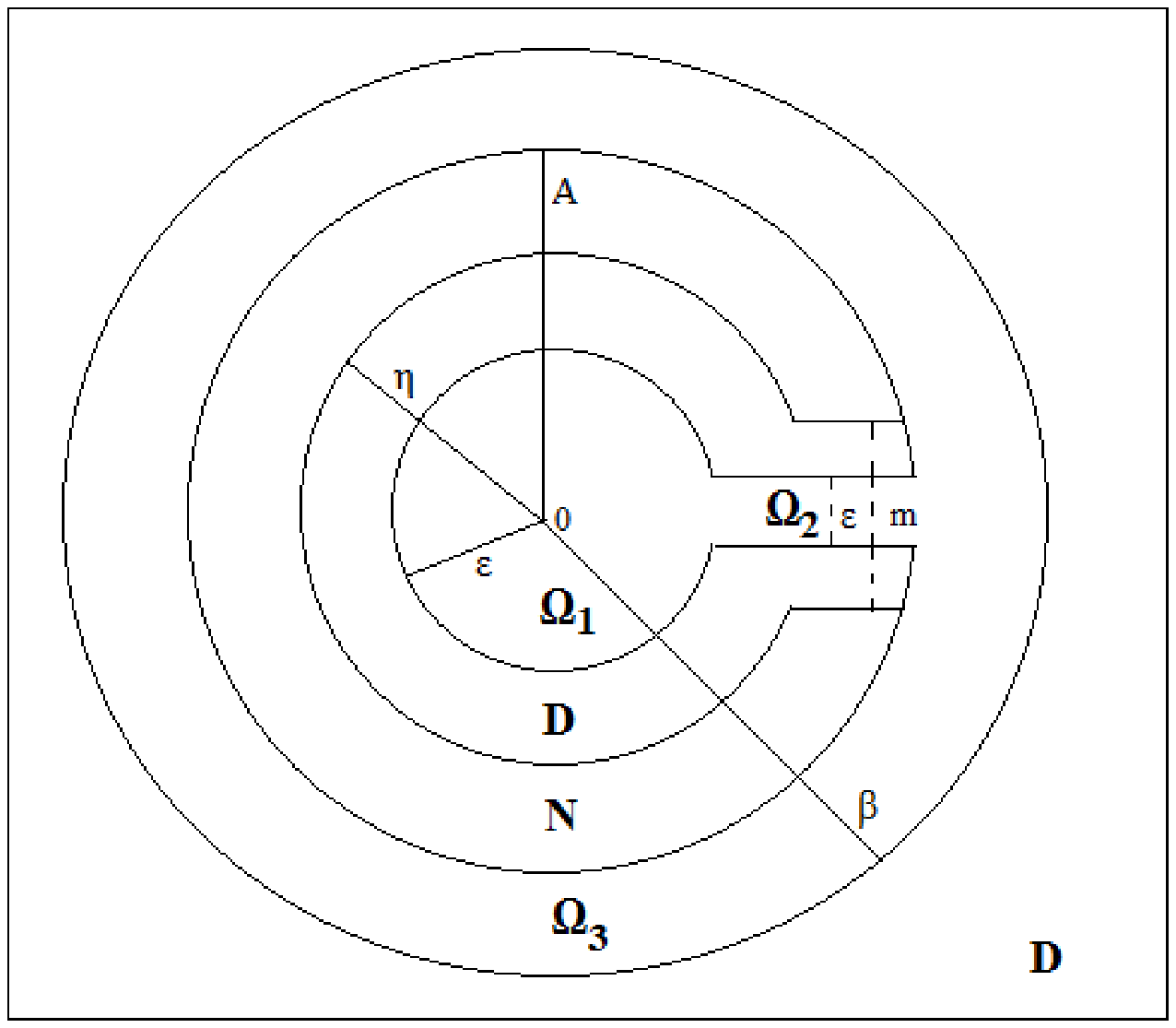}\\
\caption{Example 1}\label{Exemple}
\end{figure}
To prove that $\L_N=\L$, we will show the existence of $\e_0$ such that if $\e\le \e_0$, then $\mathcal{N}_sw(x)\ge 0$ for all $x\in N$. Notice that
\begin{eqnarray*}
\mathcal{N}_sw(x) &=& \int_{\Omega} \dfrac{(w(x)-w(y))}{|x-y|^{d+2s}}\,dy\\ &=& \int_{\Omega_1} \dfrac{(w(x)-w(y))}{|x-y|^{d+2s}}\,dy+\int_{\Omega_2}
\dfrac{(w(x)-w(y))}{|x-y|^{d+2s}}\,dy+\int_{\Omega_3} \dfrac{(w(x)-w(y))}{|x-y|^{d+2s}}\,dy\\ &=& J_1+J_2+J_3.
\end{eqnarray*}
The main idea is to choose $\e$ small in order to reach the above condition. Since $x\in N$, then $\eta\le |x|\le A$.

Let us begin by estimating $J_1$. Setting $y=|y|y'$  and $x=|x|x'$, it holds that
\begin{eqnarray*}
J_1 &= & \int_{\Omega_1}\dfrac{(w(x)-w(y))}{|x-y|^{d+2s}}\,dy=\int_{B(0,\e)}\dfrac{|x|^{-\a_0}-|y|^{-\a_0}}{|x-y|^{d+2s}}\,dy\\ &=&
\int_0^\e(|x|^{-\a_0}-\rho^{-\a_0})\rho^{d-1}\bigg(\dint\limits_{|y'|=1}\dfrac{dH^{d-1}(y')}{||x|x'-\rho y'|^{d+2s}}\bigg)d\rho
\end{eqnarray*}
where $\rho=|y|$. Let $\s=\frac{\rho}{|x|}$, then following closely the radial computation as in  \cite{FV}, it follows that
$$
J_1=\frac{1}{|x|^{2s+\a_0}}\int_{0}^{\frac{\e}{|x|}}(1-\s^{-\a_0})\sigma^{d-1}K(\sigma)d\sigma
$$
where
$$
K(\s)=\dint\limits_{|y'|=1}\dfrac{dH^{d-1}(y')}{|x'-\s y'|^{d+2s}}=2\frac{\pi^{\frac{d-1}{2}}}{\Gamma(\frac{d-1}{2})}\int_0^\pi
\frac{\sin^{d-2}(\theta)}{(1-2\sigma \cos (\theta)+\sigma^2)^{\frac{d+2s}{2}}}d\theta.
$$
Choosing $\e<<\eta$, there results that $\frac{\e}{|x|}\le \frac{\e}{\eta}<<1$, hence
$$|J_1|=\frac{1}{|x|^{\a_0+2s}}\int_{0}^{\frac{\e}{|x|}}(1-\s^{\a_0})\s^{d-\a_0-1}K(\s)d\s=o(\e).$$
We deal now with $J_2$. Without loss of generality we will assume that $\e<\min\{\frac{\eta}{4}, \frac{m}{4}\}$ and fix $\varrho=\min\{\frac{\eta}{3},
\frac{m}{3}\}$. It is clear that for all $x\in N$ and for all $y\in \O_2$, we have $|x-y|\ge \varrho$. Thus
\begin{eqnarray*}
|J_2| &\le &\int_{\Omega_2\cap |x-y|\ge \varrho}\dfrac{|w(x)-w(y)|}{|x-y|^{d+2s}}\,dy\le
\frac{C}{\varrho^{d+2s}}\int_{\Omega_2}\bigg|\eta^{-\a_0}-|y|^{-\a_0}\bigg|dy.
\end{eqnarray*}
Since $|y|^{-\a_0}\in L^1_{loc}(\re^d)$, then by the Dominated Convergence Theorem it holds $|J_2|=o(\e).$

We deal now with $J_3$. Following closely the computation of $J_1$, we reach that
$$J_3=\frac{1}{|x|^{\a_0+2s}}\int_{\frac{A}{|x|}}^{\beta}(\s^{\a_0}-1)\s^{d-\a_0-1}K(\s)d\s\ge
\frac{1}{A^{\a_0+2s}}\int_{2}^{\beta}(\s^{\a_0}-1)\s^{d-\a_0-1}K(\s)d\s.$$
Choosing $\beta>>2$ and combining the above estimates, we conclude that
$$\mathcal{N}_s(w(x))\ge \frac{1}{A^{\a_0+2s}}\int_{2}^{\beta}(\s^{\a_0}-1)\s^{d-\a_0-1}K(\s)d\s-o(\e).$$
Hence we conclude.

We have also the next example where the constant $\L_N=\L$ and then it is not attained.
\begin{figure}[h!]
\centering
\includegraphics*[width=12cm, height=6cm]{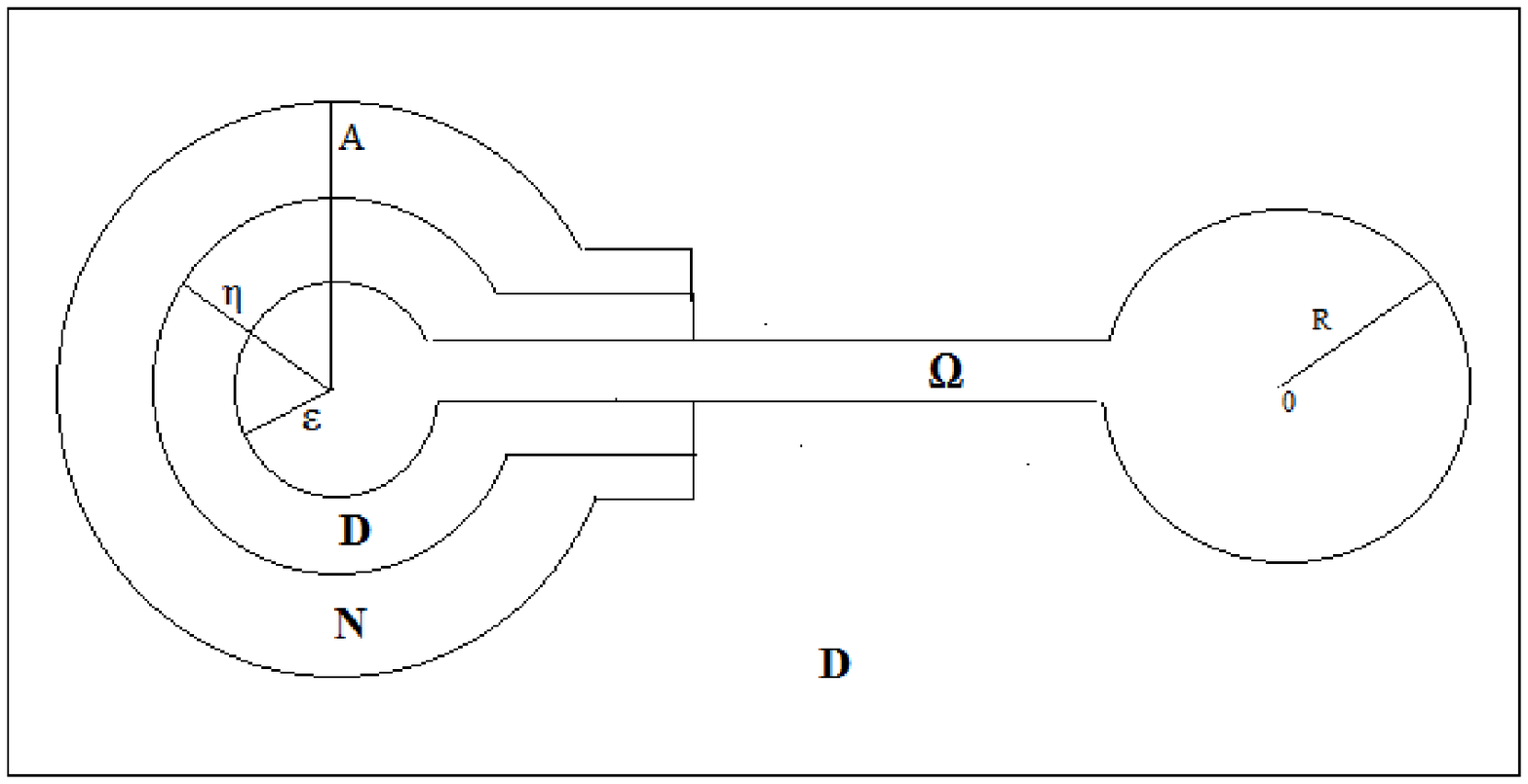}\\
\caption{Example 2}
\end{figure}

\

\section{Semilinear mixed problem involving the Hardy potential}
In this section we assume that $\Lambda_N<\Lambda$, that is, $\L_N$, is the principal eigenvalue for the corresponding mixed problem. We will consider the following nonlinear problem
\begin{equation}\label{bifur}
\left\{
\begin{array}{rcll}
(-\Delta)^s u & = & \lambda \dfrac{u}{|x|^{2s}}+ u^{p} & {\text{ in }}\O,\\
   u & > & 0 &{\text{ in }} \Omega, \\
   \mathcal{B}_{s}u & = & 0 &{\text{ in }} \RR^{d}\backslash \O,
\end{array}\right.
\end{equation}
where $1<p\le 2^*-1$ and $\l<\L_N$.

\subsection{Subcritical problems,  $1<p< 2^*-1$}
The next result is a direct consequence of Theorem \ref{nonauto} and the classical Rabinowitz bifurcation Theorem, see \cite{Rabin}.
\begin{Theorem}
Assume that the above hypotheses hold, the problem \eqref{bifur} has an unbounded branch $\Sigma$ of positive solutions bifurcating from $(0, \L_N)$.
\end{Theorem}

More interesting is the following problem. Assume now that $\l\in (\L_{N},\L)$ and define
\begin{equation}\label{constant}
\begin{array}{lll}
I_{\l,p} &= & \inf_{\{\phi\in \mathbb{E}^{s}(\Omega, D), \phi\neq
0\}}\dfrac{\dfrac{a_{d,s}}{2}\displaystyle\int\int_{\mathcal{D}_{\Omega}}|\phi(x)-\phi(y)|^{2}\,d\nu- \l\io \frac{\phi^2}{|x|^{2s}}\,dx}{\bigg(\dyle\int_\O
|\phi|^{p+1}dx\bigg)^{\frac{2}{p+1}}},
\end{array}
\end{equation}
where $p\in (1,2^*_s-1)$. It is clear that $I_{\l,p}<0$, however we have the next result.
\begin{Theorem}\label{th:diff}
Assume that $\l\in (\L_{N},\L)$ and $1<p<2^*_s-1$, then $I_{\l,p}<0$, is finite and it is achieved. Hence the problem
\begin{equation}\label{pm01}
\left\{
\begin{array}{rcll}
(-\Delta)^s u + u^{p}& = & \lambda \dfrac{u}{|x|^{2s}} & {\text{ in }}\O,\\
   u & > & 0 &{\text{ in }} \Omega, \\
   \mathcal{B}_{s}u & = & 0 &{\text{ in }} \RR^{d}\backslash \O,
\end{array}\right.
\end{equation}
has a positive solution.
\end{Theorem}
\begin{proof}
We divide the proof into two  steps.

{\sc Step 1: $|I_{\l,p}|<\infty$.}

Let $u\in \mathbb{E}^{s}(\Omega, D)$ be such that $\dyle\io |u|^{p+1}\,dx=1$ and consider $\psi\in C_0^{\infty}(B_\rho(0))$ to be a cut-off function such that $0\le
\psi\le 1$ and $\psi=1$ in $B_{\rho/4}(0)$ for $\rho>0$ small enough, then $u=\psi u +(1-\psi)u=u_{1}+u_{2}$.

Since $p+1>2$, then
\begin{equation}\label{control0}
\io \dfrac{u^2}{|x|^{2s}}dx\le\io \dfrac{u^2_1}{|x|^{2s}}dx+C(\O). \end{equation} On the other hand we have
\begin{eqnarray*}
(u(x)-u(y))^{2} &= & (u_{1}(x)-u_{1}(y))^{2}+(u_{2}(x)-u_2(y))^{2}\\ &+& 2(u_1(x)-u_1(y))(u_2(x)-u_2(y)).
\end{eqnarray*}
Then
\begin{equation}\label{control1}
\begin{array}{lll}
\displaystyle\int\int_{\mathcal{D}_{\Omega}}(u(x)-u(y))^{2}d\nu &= &
\displaystyle\int\int_{\mathcal{D}_{\Omega}}(u_{1}(x)-u_{1}(y))^{2}d\nu+\displaystyle\int\int_{\mathcal{D}_{\Omega}}(u_{2}(x)-u_2(y))^{2}d\nu\\ &+&
2\displaystyle\int\int_{\mathcal{D}_{\Omega}}(u_{1}(x)-u_{1}(y))(u_{2}(x)-u_2(y))d\nu.
\end{array}
\end{equation}
We estimate the last integral. By a direct computation it holds that
\begin{eqnarray*}
&\displaystyle\int\int_{\mathcal{D}_{\Omega}}(u_{1}(x)-u_{1}(y))(u_{2}(x)-u_2(y))d\nu\\ &=\displaystyle\io\io (u_{1}(x)-u_{1}(y))(u_{2}(x)-u_2(y)) d\nu
+2\int_{\O^c}\int_\O (u_{1}(x)-u_{1}(y))(u_{2}(x)-u_2(y))d\nu\\ &=K_{1}+K_{2}.
\end{eqnarray*}
By the elementary identity,
\begin{eqnarray*}
&\bigg(u_{1}(x)-u_{1}(y)\bigg)\bigg(u_{2}(x)-u_2(y)\bigg)=\\
&\psi(x)\bigg(1-\psi(x)\bigg)\bigg(u(x)-u(y)\bigg)^2+\bigg(1-2\psi(x)\bigg)u(y)\bigg(u(x)-u(y)\bigg)\bigg(\psi(x)-\psi(y)\bigg)\\ &
-u^2(y)\bigg(\psi(x)-\psi(y)\bigg)^2,
\end{eqnarray*}
and using Young inequality we obtain that
$$
K_{1}\ge -\e \displaystyle\io\io(u(x)-u(y))^{2}d\nu -C(\e) \io\io u^2(y)(\psi(x)-\psi(y))^2d\nu.
$$
Notice that
$$
\io\io u^2(y)(\psi(x)-\psi(y))^2d\nu\le C\io\io \frac{u^2(y)}{|x-y|^{d+2s-2}}dxdy\le C\io u^2(y)dy\le C(\O),
$$
where we have used the fact that $\sup_{\{x\in \O\}}\dyle\io \frac{1}{|x-y|^{d+2s-2}}dx\le C(\O)$. Thus
\begin{equation}\label{control2}
K_{1}\ge -C(\O,\e)-\e \displaystyle\io\io(u(x)-u(y))^{2}d\nu.
\end{equation}
Let analyze now $K_{2}$. It is clear that
\begin{eqnarray*}
K_{2} &= & \int_{\O^c}\int_{B_{2\rho}(0)}\psi(x)u(x)\bigg((1-\psi(x))u(x)-u(y)\bigg)d\nu\\ &=&
\int_{\O^c}\int_{B_{2\rho}(0)}\psi(x)(1-\psi(x))u^2(x)d\nu-\int_{\O^c}\int_{B_{2\rho}(0)}\psi(x)u(x)u(y)d\nu
\end{eqnarray*}
Since
\begin{equation}\label{control00}
\sup_{\{x\in B_{2\rho}(0)\}}\dyle\int_{\O^c} \frac{dy}{|x-y|^{d+2s}}\le C(\O,B_{2\rho}(0)),
\end{equation}
then
$$
\int_{\O^c}\int_{B_{2\rho}(0)}\psi(x)(1-\psi(x))u^2(x)d\nu\le C\int_{B_{2\rho}(0)}u^2(x)dx\le C(\O,B_{2\rho}(0)).
$$
Now, using Young inequality and  the estimate \eqref{control00}, we reach that
\begin{eqnarray*}
\int_{\O^c}\int_{B_{2\rho}(0)}\psi(x)u(x)u(y)d\nu &\le & \e \int_{\O^c}\int_{B_{2\rho}(0)}u^2(y)d\nu+ C(\e) \int_{\O^c}\int_{B_{2\rho}(0)}\psi^2(x)u^2(x)d\nu\\
&\le & \e \int_{\O^c}\int_{B_{2\rho}(0)}u^2(y)d\nu+C(\O,B_{2\rho}(0), \e).
\end{eqnarray*}
Since $u^2(y)\le 2(u(x)-u(y))^2+2u^2(x)$, it follows that
\begin{eqnarray*}
 \e\int_{\O^c}\int_{B_{2\rho}(0)}u^2(y)d\nu &\le &  2\e \int_{\O^c}\int_{B_{2\rho}(0)}(u(x)-u(y))^2\,d\nu + 2\e\int_{\O^c}\int_{B_{2\rho}(0)}u^2(x)d\nu\\
&\le & 2\e \int_{\O^c}\int_{B_{2\rho}(0)}(u(x)-u(y))^2\,d\nu + C(\O,B_{2\rho}(0), \e).
\end{eqnarray*}
Thus
\begin{equation}\label{control3}
K_{2}\ge -2\e \int_{\O^c}\int_{B_{2\rho}(0)}(u(x)-u(y))^2\,d\nu -C(\O,B_{2\rho}(0), \e).
\end{equation}
Therefore combining estimates \eqref{control0}, \eqref{control1}, \eqref{control2} and \eqref{control3}, we conclude that
\begin{eqnarray*}
& & \dfrac{a_{d,s}}{2}\displaystyle\int\int_{\mathcal{D}_{\Omega}}(u(x)-u(y))^{2}\,d\nu- \l\io \frac{u^2}{|x|^{2s}}\,dx\\ & & \ge
\bigg(\dfrac{a_{d,s}}{2}\displaystyle\int\int_{\mathcal{D}_{\Omega}}(u_1(x)-u_1(y))^{2}\,d\nu- \l\io \frac{u^2_1}{|x|^{2s}}\,dx\bigg) +
\dfrac{a_{d,s}}{2}\displaystyle\int\int_{\mathcal{D}_{\Omega}}(u_2(x)-u_2(y))^{2}\,d\nu\\ & & -3\e
\dfrac{a_{d,s}}{2}\displaystyle\int\int_{\mathcal{D}_{\Omega}}(u(x)-u(y))^{2}\,d\nu -C(\O,B_{2\rho}(0), \e).
\end{eqnarray*}
Since $u_1\in H^s_0(\O)$ and $\l<\L$, then
$$
\dfrac{a_{d,s}}{2}\displaystyle\int\int_{\mathcal{D}_{\Omega}}(u(x)-u(y))^{2}\,d\nu- \l\io \frac{u^2}{|x|^{2s}}\,dx\ge
(\L-\l)\dfrac{a_{d,s}}{2}\displaystyle\int\int_{\mathcal{D}_{\Omega}}(u(x)-u(y))^{2}\,d\nu.
$$
Thus
\begin{eqnarray*}
& & \dfrac{a_{d,s}}{2}\displaystyle\int\int_{\mathcal{D}_{\Omega}}(u(x)-u(y))^{2}\,d\nu- \l\io \frac{u^2}{|x|^{2s}}\,dx\\ & &\ge
C(\l,\L)\bigg(\displaystyle\int\int_{\mathcal{D}_{\Omega}}(u_2(x)-u_2(y))^{2}\,d\nu +\int\int_{\mathcal{D}_{\Omega}}(u_1(x)-u_1(y))^{2}\,d\nu\bigg)\\ & & -3\e
\dfrac{a_{d,s}}{2}\displaystyle\int\int_{\mathcal{D}_{\Omega}}(u(x)-u(y))^{2}\,d\nu -C(\O,B_{2\rho}(0), \e).
\end{eqnarray*}
Choosing $\e$ small, we reach that
\begin{equation}\label{coercive}
\begin{array}{lll}
& & \dfrac{a_{d,s}}{2}\displaystyle\int\int_{\mathcal{D}_{\Omega}}(u(x)-u(y))^{2}\,d\nu- \l\io \frac{u^2}{|x|^{2s}}\,dx\\ &  &\ge  C(\l,\L,
\e)\displaystyle\int\int_{\mathcal{D}_{\Omega}}(u(x)-u(y))^{2}\,d\nu- C(\O,B_{2\rho}(0), \e).
\end{array}
\end{equation}
Thus $|I_{\l,p}|<\infty$.

{\sc Step 2: $I_{\l,p}$ is attained.} Define
$$ I_{\l,n}=\inf_{\{\phi\in \mathbb{E}^{s}(\Omega, D), \phi\neq
0\}}\dfrac{\dfrac{a_{d,s}}{2}\displaystyle\int\int_{\mathcal{D}_{\Omega}}|\phi(x)-\phi(y)|^{2}\,d\nu- \l\io \frac{\phi^2}{|x|^{2s}+\frac
1n}\,dx}{\bigg(\dyle\int_\O |\phi|^{p+1} dx\bigg)^{\frac{2}{p+1}}},
$$ it is clear that $I_{\l,n}\downarrow I_{\l,p}$ as $n\to \infty$.
Hence $I_{\l,n}<0$ for $n\ge n_0$.

Since $p+1<2^*_s$, then using a variational argument we get that $I_{\l,n}$ is achieved. Hence we get the existence of $u_n\in \mathbb{E}^{s}(\Omega, D)$ that
satisfies
$$
(P_n)\equiv \left\{
\begin{array}{rcll}
(-\Delta)^s u_n -\lambda \dfrac{u_n}{|x|^{2s}+\frac 1n} &=& I_{\l,n}u^{p}_n & {\text{ in }}\O,\\
   u_n & \ge & 0 &{\text{ in }} \Omega, \\
   \mathcal{B}_{s}u_n & = & 0 &{\text{ in }} \re^{d}\backslash \O,
\end{array}\right.
$$ with $\|u_n\|_{L^{p+1}(\O)}=1$

We claim that $\{u_n\}_n$ is bounded in the space $ \mathbb{E}^{s}(\Omega, D)$. Since
\begin{equation*}
\begin{array}{lll}
& \dfrac{a_{d,s}}{2}\displaystyle\int\int_{\mathcal{D}_{\Omega}}(u_n(x)-u_n(y))^{2}\,d\nu- \l\io \frac{u^2_n}{|x|^{2s}+\frac 1n}\,dx\ge \\ &
\dfrac{a_{d,s}}{2}\displaystyle\int\int_{\mathcal{D}_{\Omega}}(u_n(x)-u_n(y))^{2}\,d\nu- \l\io \frac{u^2_n}{|x|^{2s}}\,dx,
\end{array}
\end{equation*}
then by \eqref{coercive}, it follows that
\begin{equation*}
\begin{array}{lll}
& & \dfrac{a_{d,s}}{2}\displaystyle\int\int_{\mathcal{D}_{\Omega}}(u_n(x)-u_n(y))^{2}\,d\nu- \l\io \frac{u^2_n}{|x|^{2s}+\frac 1n}\,dx\\ &  &\ge  C(\l,\L,
\e)\displaystyle\int\int_{\mathcal{D}_{\Omega}}(u_n(x)-u_n(y))^{2}\,d\nu- C(\O,B_{2\rho}(0), \e).
\end{array}
\end{equation*}
Thus $\displaystyle\int\int_{\mathcal{D}_{\Omega}}(u_n(x)-u_n(y))^{2}\,d\nu\le C$ for all $n$ and the claim follows.

Therefore, there exists  $u_0\in \mathbb{E}^{s}(\Omega, D)$ such that $u_n\rightharpoonup u_0$ weakly in $\mathbb{E}^{s}(\Omega, D)$ and strongly in
$L^{p+1}(\O)$. Hence $\|u_0\|_{L^{p+1}(\O)}=1$ and then $u_0\not\equiv 0$. Notice that by the weak convergence we obtain that $u_0$ is a weak solution to
\eqref{pm01}.

We claim now that $\dfrac{u^2_n}{|x|^{2s}}\to \dfrac{u^2_0}{|x|^{2s}}$ strongly in $L^1(\O)$. Define $w_n=u_n-u_0$, it is clear that $w_n\rightharpoonup 0$ weakly
in $\mathbb{E}^{s}(\Omega, D)$ and $w_n\to 0$ strongly in $L^{p+1}(\O)$.

As in the previous step, we have
$$\displaystyle\int\int_{\mathcal{D}_{\Omega}}(w_n(x)-w_n(y))^{2}\,d\nu\ge
\displaystyle\int\int_{\mathcal{D}_{\Omega}}\bigg((\psi w_n)(x)-(\psi w_n)(y)\bigg)^{2}\,d\nu +o(1)
$$ and $$ \io \dfrac{w_n^2}{|x|^{2}+\frac{1}{n}}dx=\io
\dfrac{(\psi w_n)^2}{|x|^{2}+\frac{1}{n}}dx+o(1). $$
Since $w_n\in H^s_0(\O)$ and  $u_0$ and $u_n$ are solution of the problems \eqref{pm01} and $(P_n)$ respectively, it holds that
\begin{eqnarray*}
o(1) &\ge & \dfrac{a_{d,s}}{2}\displaystyle\int\int_{\mathcal{D}_{\Omega}}(w_n(x)-w_n(y))^{2}\,d\nu- \l\io \frac{w_n^2}{|x|^{2s}}\,dx\\ & \ge  &
\dfrac{a_{d,s}}{2}\displaystyle\int\int_{\mathcal{D}_{\Omega}}((\psi w_n)(x)-(\psi w_n)(y))^{2}\,d\nu-\l\io \dfrac{(\psi w_n)^2}{|x|^{2s}+\frac{1}{n}}dx +o(1)\\
&\ge & (\L-\l)\dyle \io \frac{w_n^2}{|x|^{2s}}\,dx +o(1).
\end{eqnarray*}
Hence $\dyle\io \frac{w_n^2}{|x|^{2s}}\,dx=o(1)$ and the claim follows.

Combing the above estimates we reach that $u_n\to u_0$ strongly in $\mathbb{E}^{s}(\Omega, D)$ and thus $u_0$ realize $I_{\l,p}$. Hence up to a positive constant,
$cu_0$ solves problem \eqref{pm01}, then we conclude.

\end{proof}

\subsection{Doubly-Critical problem }\label{sec:critical}
In this subsection we discuss the existence and the non existence to the following double critical problem
\begin{equation}\label{P2}
\left\{
\begin{array}{rcll}
(-\Delta)^s u & = & \lambda \dfrac{u}{|x|^{2s}}+ u^{2^*_s-1} & {\text{ in }}\O,\\
   u & > & 0 &{\text{ in }} \Omega, \\
   \mathcal{B}_{s}u & = & 0 &{\text{ in }} \RR^{d}\backslash \O,
\end{array}\right.
\end{equation}
where $\l\in (0, \L_N)$ according to the D-N configuration. If $\O=\re^d$, problem \eqref{P2} is related to the next constant
\begin{equation}\label{minm1}
S_\l=\inf\limits_{\{u\in \mathcal{C}^\infty_0(\mathbb{R}^d)  , ||u||\neq 0,\,||u||_{2^*_s}=1  \}} \dfrac{a_{d,s}}{2}\dyle\int_{\mathbb{R}^d}\int_{\mathbb{R}^d} |u(x)-u(y)|^2\,d\nu
-\dyle\l\int_{\mathbb{R}^d}\dfrac{u^2(x)}{|x|^{2s}}dx.
\end{equation}
The problem  in the whole euclidian space  $\mathbb{R}^d$  has been studied in \cite{DMPS}. From the result of \cite{DMPS} we know that the constant $S_\l$ is
independent of $\O$ containing the pole of the Hardy potential.

In the same way we consider for a D-N configuration the constant $T_{\l,N}$ defined by
\begin{equation}\label{minm2}
T_{\l,N}=\inf\limits_{\{u\in E^s(\O,D), ||u||\neq 0,\,||u||_{2^*_s}=1  \}} \dfrac{a_{d,s}}{2}\dyle\int\int_{\mathcal{D}_{\Omega}} |u(x)-u(y)|^2\,d\nu
 -\l\int_{\Omega}\dfrac{u^2(x)}{|x|^{2s}}dx.
\end{equation}
It is clear that if $T_{\l,N}$ is achieved, then problem \eqref{P2} has a nontrivial solution.

We have the next existence result.
\begin{Theorem}\label{exi-critical} Let $(D,N)$ a D-N configuration and assume that $\l\in (0, \L_N)$. Suppose that $T_{\l,N}<\min\{S_\l, S_N\}$, then $T_{\l,N}$ is achieved and, as a consequence, problem \eqref{P2} has a nontrivial solution.
\end{Theorem}
\begin{proof}
Recall that $S_N$ is the Sobolev constant defined in Proposition \ref{sobolev}. Since $\l<\L_N$, then $T_{\l,N}\ge (1-\frac{\l}{\L_N})S_N>0$.

Let $\{u_n\}_n\subset \mathbb{E}^{s}(\O, D)$  be a minimizing sequence for $T_{\l,N}$ with $\dyle \int_{\O}|u_n|^{2^*_s}\, dx=1$, then $\{u_n\}_n$ is bounded in
$\mathbb{E}^{s}(\O,D)$, and
$$
\dfrac{a_{d,s}}{2}\dyle\int\int_{\mathcal{D}_{\Omega}} |u_n(x)-u_n(y)|^2\,d\nu
 -\l\int_{\Omega}\dfrac{u^2_n(x)}{|x|^{2s}}dx\to T_{\l,N}.
 $$
 Without loss of generality we can choose $u_n\ge 0$ in $\re^d$.
Hence there exists $\bar{u}\in \mathbb{E}^{s}(\O,D)$ such that $u_n \rightharpoonup \bar{u} $ weakly in $\mathbb{E}^{s}(\O,D)$, and up to a subsequence,
$u_n\to \bar{u}$ strongly in $L^\s(\O)$ for all $\s<2^*_s$ and $u_n\to \bar{u}$ a.e in $\O$.

Using the Ekeland variational principle it holds that
\begin{equation}\label{Pcrit}
\left\{
\begin{array}{rcll}
(-\Delta)^s u_n -\l\dfrac{u_n}{|x|^{2s}} &=& T_{\l,N}\,u_n^{2^*_s-1}+o(1) & {\text{ in }}\O,\\
   \mathcal{B}_{s}u_n & = & 0 &{\text{ in }} \re^{d}\backslash \O.
\end{array}\right.
\end{equation}
It is clear that if $\bar{u}\neq 0$, then $\bar{u}$ solves the problem \eqref{P2}.

Assume by contradiction that $\bar{u}=0$. Let $\psi\in C_0^{\infty}(B_\rho(0))$ be a cut-off function such that $0\le \psi\le 1$ and $\psi=1$ in $B_{\rho/4}(0)$
for $\rho>0$ small enough.

We claim that
\begin{equation}\label{tlm0}
\io (-\Delta)^s u_n \, u_n\psi^2 dx =\dfrac{a_{d,s}}{2}\displaystyle\int\int_{\mathcal{D}_{\Omega}}\bigg((\psi u_n)(x)-(\psi u_n)(y)\bigg)^{2}\,d\nu +o(1).
\end{equation}
Notice that
$$
\io (-\Delta)^s u_n \, u_n\psi^2 dx =\dfrac{a_{d,s}}{2}\displaystyle\int\int_{\mathcal{D}_{\Omega}}(u_n(x)-u_n(y))\bigg((\psi^2 u_n)(x)-(\psi^2
u_n)(y)\bigg)\,d\nu +o(1).
$$
Since
\begin{eqnarray*}
&(u_n(x)-u_n(y))\bigg(\psi u_n)(x)-(\psi u_n)(y)\bigg)-((\psi u_n)(x)-(\psi u_n)(y))^{2}=\\ &-u_n(x)u_n(y)(\psi(x)-\psi(y))^2,
\end{eqnarray*}
we reach that
\begin{equation}\label{tlm}
\begin{array}{lll}
&\dyle \io (-\Delta)^s u_n \, u_n\psi^2 dx =\displaystyle\int\int_{\mathcal{D}_{\Omega}}\bigg((\psi u_n)(x)-(\psi u_n)(y)\bigg)^{2}\,d\nu \\
&-\displaystyle\int\int_{\mathcal{D}_{\Omega}}u_n(x)u_n(y)(\psi(x)-\psi(y))^2d\nu.
\end{array}
\end{equation}
Notice that
\begin{eqnarray*}
&\displaystyle\int\int_{\mathcal{D}_{\Omega}}u_n(x)u_n(y)(\psi(x)-\psi(y))^2d\nu =\\ &\displaystyle\int_{\Omega}\int_{\Omega}u_n(x)u_n(y)(\psi(x)-\psi(y))^2d\nu+
2\displaystyle\int_{\O^c}\int_{\Omega}u_n(x)u_n(y)(\psi(x)-\psi(y))^2d\nu=\\ & K_{1n}+2K_{2n}.
\end{eqnarray*}
Let us begin by estimating $K_{1n}$. It is clear that
\begin{eqnarray*}
K_{1n} &= & C(\Omega)\int_{\O}\int_{\O}\frac{u_n(x)u_n(y)}{|x-y|^{d+2s-2}}dxdy\\ &\le & 2C(\O)\dyle \int_{\O}\int_{\O}\frac{u^2_n(x)}{|x-y|^{d+2s-2}}dxdy+
2C(\O)\dyle \int_{\O}\int_{\O}\frac{u^2_n(y)}{|x-y|^{d+2s-2}}dxdy.
\end{eqnarray*}
Since $\sup_{x\in \O}\dint_{\O}\frac{dy}{|x-y|^{d+2s-2}}\le C(\O)$ and $\sup_{y\in \O}\int_{\O}\frac{dx}{|x-y|^{d+2s-2}}\le C(\O)$, we reach that
$$
K_{1n}\le 4C(\O)\int_{\O}u^2_n(x)dx=o(1).
$$
Now, we deal with $K_{2n}$. We have
\begin{eqnarray*}
K_{2n} &= & \int_{\O^c}\int_{B_{\rho}(0)}u_n(x)u_n(y)\psi^2(x)d\nu\\ &\le & \bigg(\int_{\O^c}\int_{B_{\rho}(0)}u^2_{n}(x)d\nu\bigg)^{\frac 12}
\bigg(\int_{\O^c}\int_{B_{\rho}(0)}u^2_{n}(y)d\nu\bigg)^{\frac 12}.
\end{eqnarray*}
It is clear that $\displaystyle\int_{\O^c}\int_{B_{\rho}(0)}u^2_{n}(x)d\nu=o(1)$, now using the fact that $\{u_n\}_n$ is bounded in $\mathbb{E}^{s}(\Omega, D)$
and $u^2(y)\le 2(u(x)-u(y))^2+2u^2(x)$, we get
\begin{eqnarray*}
\int_{\O^c}\int_{B_{\rho}(0)}u^2_n(y)d\nu &\le &  2\int_{\O^c}\int_{B_{\rho}(0)}(u_n(x)-u_n(y))^2\,d\nu + 2\int_{\O^c}\int_{B_{\rho}(0)}u^2_n(x)d\nu\\ &\le & C.
\end{eqnarray*}
Hence $K_{2n}=o(1)$.

Combining the above estimate and going back to \eqref{tlm}, we reach \eqref{tlm0} and the claim follows.

Therefore using $u_n\psi^2$ as a test function in \eqref{Pcrit}, we conclude that
\begin{eqnarray*}
\dfrac{a_{d,s}}{2}\displaystyle\int\int_{\mathcal{D}_{\Omega}}\bigg((\psi u_n)(x)-(\psi u_n)(y)\bigg)^{2}\,d\nu-\l\io \dfrac{(\psi u_n)^2}{|x|^{2s}}dx &= &
T_{\l,N}\io u_n^{2^*_s}\psi^2 dx +o(1)\\ & \le &  T_{\l,N}\bigg(\io (\psi\,u_n)^{2^*_s}dx\bigg)^{\frac{2}{2^*_s}} +o(1).
\end{eqnarray*}
We set $u_{1n}=u_n\psi$, then $u_{1n}\in H^s_0(\O)$. If for a subsequence of $\{u_n\}_n$, we have $\dyle\io u^{2^*_s}_{1n} dx\ge C$, then we conclude that
$$
S_\l\le \dfrac{\dfrac{a_{d,s}}{2}\displaystyle\int\int_{\mathcal{D}_{\Omega}}(u_{1n}(x)-u_{1n}(y))^2\,d\nu-\l\io
\dfrac{u_{1n}^2}{|x|^{2s}}dx}{\bigg(\displaystyle\io u_{1n}^{2^*_s}dx\bigg)^{\frac{2}{2^*_s}}}\le T_{\l,N}+o(1).
$$
Thus $S_\l\le T_{\l,N}$ which is a contradiction with the  hypothesis in the statement of the Theorem.

Hence $\displaystyle\io u^{2^*_s}_{1n} dx\to 0$ as $n\to \infty$. Thus $\displaystyle\io u^{2^*_s}_{n}(1-\psi)^{2^*_s} dx\to 1$ as $n\to \infty$.

We set $\varrho=1-\psi$, by the computation above, we  reach that
\begin{equation}\label{gran0}
\io (-\Delta)^s u_n \, u_n\varrho^2 dx =\dfrac{a_{d,s}}{2}\displaystyle\int\int_{\mathcal{D}_{\Omega}}\bigg((\varrho^2 u_n)(x)-(\varrho^2 u_n)(y)\bigg)^{2}\,d\nu
+o(1),
\end{equation}
and
$$
\io \dfrac{(\varrho u_n)^2}{|x|^{2s}}dx =o(1).
$$
Thus using $u_n\varrho^2$ as a test function in \eqref{Pcrit} we conclude that
\begin{eqnarray*}
\dfrac{a_{d,s}}{2}\displaystyle\int\int_{\mathcal{D}_{\Omega}}\bigg((\varrho^2 u_n)(x)-(\varrho^2 u_n)(y)\bigg)^{2}\,d\nu &= & T_{\l,N}\io u_n^{2^*_s}\varrho^2 dx
+o(1)\\ & \le &  T_{\l,N}\bigg(\io (\varrho\,u_n)^{2^*_s}dx\bigg)^{\frac{2}{2^*_s}} +o(1).
\end{eqnarray*}
Setting $u_{2n}=u_n\varrho$, it follows that
$$
S_N\le \dfrac{\dfrac{a_{d,s}}{2}\displaystyle\int\int_{\mathcal{D}_{\Omega}}(u_{2n}(x)-u_{2n}(y))^2\,d\nu}{\bigg(\displaystyle\io
u_{2n}^{2^*_s}dx\bigg)^{\frac{2}{2^*_s}}}\le T_{\l,N}+o(1).
$$
Thus $S_N\le T_{\l,N}$ which again  is a contradiction with the  hypothesis.

Hence, as a conclusion we obtain that $u_0\neq 0$. It is clear that, up to a constant, $u_0$ solves problem \eqref{P2}.

To finish we have just to show that $u_0$ realize $T_{\l,N}$. Let
$$Q_{\l,N}(u)\equiv \dfrac{a_{d,s}}{2}\dyle\int\int_{\mathcal{D}_{\Omega}} |u(x)-u(y)|^2\,d\nu
 -\l\int_{\Omega}\dfrac{u^2(x)}{|x|^{2s}}dx,$$
since $\l<\L_N$, then $Q_{\l,N}$ define an equivalent norm to the norm of the space $\mathbb{E}^{s}(\Omega, D)$ hence, to conclude, we have just to show that $
Q_{\l,N}(u_n-u_0)\to 0$ as $n\to \infty$. Recall that $u_n\rightharpoonup u_0$ weakly in $\mathbb{E}^{s}(\Omega, D)$, then
\begin{equation}\label{TR}
Q_{\l,N}(u_n)=Q_{\l,N}(u_0)+Q_{\l,N}(u_n-u_0)+o(1).
\end{equation}
In the same way, using Brezis-Lieb Lemma, we get
$$
||u_n||_{L^{2^*_s}(\O)}=||u_0||_{L^{2^*_s}(\O)}+||u_n-u_0||_{L^{2^*_s}(\O)}+o(1).
$$
Since $Q_{\l, N}(u_n-u_0)\ge S_{\l,N}||u_n-u_0||^{\frac{2}{2^*_s}}_{L^{2^*_s}(\O)}$, then
\begin{equation}\label{last00}
\begin{array}{lll}
\dfrac{Q_{\l,N}(u_0)}{||u_0||^{\frac{2}{2^*_s}}_{L^{2^*_s}(\O)}} &=
&\dfrac{Q_{\l,N}(u_n)-Q_{\l,N}(u_n-u_0)+o(1)}{\bigg(||u_n||_{L^{2^*_s}(\O)}-||u_n-u_0||^{\frac{2}{2^*_s}}_{L^{2^*_s}(\O)}\bigg)^{\frac{2}{2^*_s}}}\\ &\le &
S_{\l,N} \dfrac{Q_{\l,N}(u_n)-Q_{\l,N}(u_n-u_0)+o(1)}{\bigg(Q^{\frac{2^*_s}{2}}_{\l,N}(u_n)-Q^{\frac{2^*_s}{2}}_{\l,N}(u_n-u_0)+o(1)\bigg)^{\frac{2}{2^*_s}}}.
\end{array}
\end{equation}
If $\limsup_{n\to \infty}Q_{\l,N}(u_n-u_0)\neq 0$ then, by \eqref{TR}, it holds
$$\limsup_{n\to \infty}Q_{\l,N}(u_n-u_0)=S_{\l,N}-Q_{\l,N}(u_0).$$ Then going back to \eqref{last00}, it follows that
$$
S_{\l,N}\le \dfrac{Q_{\l,N}(u_0)}{||u_0||^{\frac{2}{2^*_s}}_{L^{2^*_s}(\O)}}<S_{\l,N}
$$
which leads to a contradiction. Hence $\limsup_{n\to \infty}Q_{\l,N}(u_n-u_0)=0$ and then $u_0$ realize $T_{\l,N}$.
\end{proof}

To complete the paper we give an  example where the constant $T_{\l,N}$ is realized and then problem \eqref{P2} has a positive solution.
\begin{Proposition}\label{fin} Consider  a D-N configuration such that $\L_N<\L$, then there exists $0<\bar{\l}<\L_N$ such that for all $\l\in (\bar{\l}, \L_N)$, the constant $T_{\l,N}$ is achieved.
\end{Proposition}
\begin{proof}
Since $\L_N<\L$, by Theorem \ref{exist00}, it follows that $\L_N$ is achieved. Hence  there exists $\bar{u}\in E^s(\O,D)$ such that $\bar{u}$ is a positive solution to eigenvalue problem \eqref{otov}. Without loss of generality we can assume that $||\bar{u}||_{L^{2^*_s}(\O)}=1$, thus, using the definition of $T_{\l,N}$, we conclude that
$$
T_{\l,N}\le \dfrac{a_{d,s}}{2}\dyle\int\int_{\mathcal{D}_{\Omega}} |\bar{u}(x)-\bar{u}(y)|^2\,d\nu
 -\l\int_{\Omega}\dfrac{\bar{u}^2(x)}{|x|^{2s}}dx=(\L_N-\l)\int_{\Omega}\dfrac{\bar{u}^2(x)}{|x|^{2s}}dx.
$$
Since $\l<\L_N<\L$,  by using the definition of $S_\l$, it follows that $S_\l>S_{\L_N}\ge 0$.  It is clear that
$$
(\L_N-\l)\int_{\Omega}\dfrac{\bar{u}^2(x)}{|x|^{2s}}dx\to 0\mbox{  as  }\l\to \L_N.
$$
Therefore,  there exists $\bar{\l}\equiv\bar{\l}(S_{\L_N}, S_N, \bar{u})<\L_N$ such that if $\l\in (\bar{\l}, \L_N)$, then
$$
T_{\l,N}\le (\L_N-\bar{\l})\int_{\Omega}\dfrac{\bar{u}^2(x)}{|x|^{2s}}dx<\min\{S_{\L_N}, S_N\}<\min\{S_\l, S_N\}.
$$
Hence by Theorem \ref{exi-critical}, $T_{\l,N}$ is achieved and the result follows.
\end{proof}
\begin{remark}
 The  result in Proposition \ref{fin} goes in the spirit of the classical existence result by  H. Brezis and L. Nirenberg in the seminal paper for  the Dirichlet  problem of the Laplacian  with critical exponent; that result was   extended in \cite{SV} to the nonlocal case.
  More general conditions for the existence in the doubly critical problem with mixed boundary conditions seem to be unknown.
\end{remark}

\end{document}